\documentclass[11pt,a4paper]{amsart}

\usepackage[english]{babel} 
\usepackage[T1]{fontenc} 
\usepackage{amsmath}
\usepackage{amssymb} 
\usepackage{ucs} 
\usepackage[utf8x]{inputenc}
\usepackage[mathcal]{eucal} 
\usepackage[]{graphicx} 
\usepackage{latexsym}
\usepackage{amsthm} 
\usepackage{stmaryrd}
\usepackage{verbatim}
\usepackage{hyperref}

\usepackage{pgf,tikz}
\usetikzlibrary{arrows}

\newtheorem{defi}{Definition}[section] 
\newtheorem{theo}[defi]{Theorem}
 
\newtheorem{lemme}[defi]{Lemma}
\newtheorem{lemma}[defi]{Lemma}
\newtheorem{prop}[defi]{Proposition}
\newtheorem{conj}[defi]{Question}

\theoremstyle{remark} \newtheorem*{rem}{Remark}

\newcommand{\T}{\mathbb{T}} 
\newcommand{\Ss}{\mathbb{S}}
\newcommand{\proj}{\mathbb{P}}
\newcommand{\R}{\mathbb{R}}

\newcommand{\Z}{\mathbb{Z}}
\newcommand{\N}{\mathbb{N}} 

\newcommand{\h}{\mathbb{H}}
\newcommand{\e}{\varepsilon}
\newcommand{\p}{\varphi}

\newcommand{\Log}{\rm{Log}}
\newcommand{\Diff}{\rm{Diff}}
\newcommand{\Homeo}{\rm{Homeo}}

\newcommand{\Stab}{\textrm{Stab}}
\newcommand{\PSL}{\rm{PSL}}

\newcommand{\intoo}[2]{\mathopen{]}#1\,,#2\mathclose{[}}
\newcommand{\intff}[2]{\mathopen{[}#1\,,#2\mathclose{]}}
\newcommand{\intof}[2]{\mathopen{]}#1\,,#2\mathclose{]}}
\newcommand{\intfo}[2]{\mathopen{[}#1\,,#2\mathclose{[}}

\renewcommand{\tilde}{\widetilde}

\title[Area preserving circle diffeomorphisms]{Differentiable conjugacy for groups of area preserving circle diffeomorphisms} 
\author{Daniel Monclair}
\date{\today}

\thanks{Partially supported by  ANR project GR-Analysis-Geometry (ANR-2011-BS01-003-02)}

\begin{document}

\maketitle

\begin{abstract} We study groups of circle diffeomorphisms whose action on the cylinder $\mathcal C=\Ss^1\times \Ss^1\setminus \Delta$ preserves a volume form. We first show that such a group is topologically conjugate to a subgroup of $\PSL(2,\R)$, then discuss the existence of a differentiable  conjugacy.
\end{abstract}

    \setcounter{tocdepth}{1}
\tableofcontents

\section{Introduction}

\indent A well known theorem proved by Gabai and Casson-Jungreis states that a group action on the circle $\rho:\Gamma \to \Homeo(\Ss^1)$ is conjugate in $\Homeo(\Ss^1)$ to the action of a subgroup of $\PSL(2,\R)$ (where the action  is the projective action on $\Ss^1=\R\proj^1$) if and only if  the induced action on the space of distinct triples of points is proper. This condition is known as the convergence property. However, for differentiable actions $\rho : \Gamma \to \Diff(\Ss^1)$, the conjugacy is not necessarily differentiable.\\
\indent Our goal is to study the differentiability of such a conjugacy via the  diagonal  actions on product spaces. The diagonal action of $\PSL(2,\R)$ on the space of distinct couples of points $\mathcal C=\Ss^1\times \Ss^1\setminus \Delta$ preserves the volume form $\frac{4dx\wedge dy}{(x-y)^2}$. The existence of an invariant volume form on $\mathcal C$ is a notion that is invariant under conjugacy in $\Diff(\Ss^1)$. Given a representation $\rho:\Gamma \to \Diff(\Ss^1)$, we will study the link between the existence of an invariant volume form on $\mathcal C$ and conjugacy with the action of a subgroup of $\PSL(2,\R)$. \\
\indent The first result of this paper states that this condition is stronger than the convergence property:
\begin{theo} \label{topologic} Assume that $\rho: \Gamma \to \Homeo(\Ss^1)$ preserves a continuous volume form on $\mathcal C$. Then $\rho$ is conjugate in $\Homeo(\Ss^1)$ to a representation in  $\PSL(2,\R)$. \end{theo}
The proof consists in remarking that preserving a volume form on pairs of points implies preserving a distance on triples of points. 

\subsection{Fuchsian groups and generalizations}

 We identify $\PSL(2,\R)$ and its image in $\Diff(\Ss^1)$ given by the projective action on $\Ss^1\approx \R\proj^1$, and  call a group action on the circle  $\rho: \Gamma \to \Homeo(\Ss^1)$  \textbf{Fuchsian} if $\rho(\Gamma)\subset \PSL(2,\R)$ (note that we do not ask for $\rho(\Gamma)$ to be discrete, even though it will be the case in most of our examples).\\
 \indent We will say that $\rho : \Gamma \to \Homeo(\Ss^1)$ is \textbf{topologically Fuchsian} if there is $h\in \Homeo(\Ss^1)$ such that  $h^{-1} \rho(\Gamma) h \subset \PSL(2,\R)$. 
 
 \subsubsection{Differential conjugacy}  When considering  actions by diffeomorphisms, the natural notion of conjugacy is the conjugacy in the group $\Diff(\Ss^1)$. We will say that $\rho:\Gamma\to \Diff(\Ss^1)$ is \textbf{differentially Fuchsian} if  there is $h\in \Diff(\Ss^1)$ such that $h^{-1} \rho(\Gamma) h \subset \PSL(2,\R)$ (in the absence of precision, $\Diff(\Ss^1)$ denotes the group of $C^\infty$ diffeomorphisms).\\
  \indent There is no general condition under which  a topologically Fuchsian representation $\rho :\Gamma \to \Diff(\Ss^1)$ is automatically differentially Fuchsian. However, there are  two known results assuring the existence of a differential conjugacy under specific hypotheses: a theorem of Herman  on diffeomorphisms conjugate to irrational rotations, and a theorem of Ghys  on representations of surface groups.
  
\subsubsection{Area-preserving groups} We will say that an action $\rho :\Gamma \to \Diff(\Ss^1)$ is \textbf{area-preserving} if the diagonal action on $\mathcal C=\Ss^1\times\Ss^1\setminus \Delta$ preserves a smooth volume form.\\
\indent Theorem \ref{topologic} states that an area-preserving representation is topologically Fuchsian. If $h\in \Diff(\Ss^1)$ and $\rho :\Gamma \to \Diff(\Ss^1)$ preserves the volume form $\omega$ on $\mathcal C$, then $h^{-1}\rho h$ preserves the volume form $h^\star\omega$. If $h$ is only continuous, then $h^\star\omega$ is only a measure, it is not always absolutely continuous with respect to the Lebesgue measure.\\
\indent Since the action of $\PSL(2,\R)$ preserves a volume form, all differentially Fuchsian representations are area-preserving.\\
\indent We will show that under some specific hypotheses, it is an equivalence.

\begin{theo} \label{area_equiv_diff} Assume that $\rho :\Gamma \to \Diff(\Ss^1)$ satisfy (at least) one of the following conditions:
\begin{itemize} \item There is a dense orbit on $\Ss^1$.
\item $\rho(\Gamma)\subset \Diff^\omega(\Ss^1)$ and $\Gamma$ has no finite orbit on $\Ss^1$.
\item $\Gamma = \Z$, $\rho(1)\in \Diff^\omega(\Ss^1)$  and $\rho(1)$   has exactly two fixed points.
\item $\Gamma=\Z$ and $\rho(1)$ has no fixed point on $\Ss^1$.
\end{itemize}
Then $\rho$ is area-preserving if and only if it is differentially Fuchsian.
\end{theo}

The proof is obtained by combining Proposition \ref{elliptic}, Theorem \ref{analytic_hyperbolic}, Theorem \ref{analytic_non_elementary} and Theorem \ref{transitive}. We will also see that this equivalence is not always true.

\subsubsection{$L$-Differential conjugacy}
A  group $\Gamma\subset \Homeo(\Ss^1)$ with no finite orbit has a unique minimal closed invariant set $L_{\Gamma}\subset\Ss^1$, called the limit set. It is  either the whole circle or a Cantor set. In the latter case, we call $L_\Gamma$ an exceptional minimal  set. Examples of such groups are given by Schottky groups (free groups in $\PSL(2,\R)$ generated by appropriately chosen hyperbolic elements). In this case, we will show that area-preserving actions are not necessarily differentially Fuchsian.\\
\indent However, the examples that we will give share a property with minimal actions (i.e. all orbits on $\Ss^1$ are dense): the conjugacy is always differentiable along the limit set.
\begin{defi} We  say that two representations $\rho_1,\rho_2:\Gamma\to\Diff(\Ss^1)$ with no finite orbits are $L$-differentially conjugate if there is $h\in \Homeo(\Ss^1)$ such that $h^{-1} \rho_2 h =\rho_1$ and such that there is $\p\in \Diff(\Ss^1)$ with the same  restriction  $\p_{/L_{\rho_1(\Gamma)}} = h_{/L_{\rho_1(\Gamma)}} $.\\ We say that $\rho:\Gamma\to \Diff(\Ss^1)$ is $L$-differentially Fuchsian if it is $L$-differentially conjugate to a Fuchsian action. \end{defi}
 Knowing that $L$-differentially Fuchsian actions are not necessarily differentially Fuchsian, the following statement shows that area-preserving actions are not necessarily differentially Fuchsian.
 
 \begin{theo}\label{area_preserving} If $\rho:\Gamma \to \Diff(\Ss^1)$ is  $L$-differentially conjugate to a convex cocompact representation in $\PSL(2,\R)$, then $\rho$ is area-preserving.\end{theo}

\subsubsection{Spectral conditions}
Finally, a weaker generalization of Fuchsian actions consists in looking only at the derivatives at fixed points. A hyperbolic element $\gamma\in\PSL(2,\R)$ has  exactly two fixed points $N,S\in \Ss^1$. The derivatives satisfy $\gamma'(N)\gamma'(S)=1$ and $\gamma'(N)\ne 1$.

\begin{defi} We say that $\rho:\Gamma\to\Diff(\Ss^1)$ is spectrally Möbius-like if non trivial elements have at most two fixed points, and if elements $\gamma$ with two fixed points $N,S$ satisfy $\rho(\gamma)'(N)\rho(\gamma)'(S)=1$ and $\rho(\gamma)'(N)\ne 1$. \end{defi}

This is a condition that concerns individual elements of the group rather than the group structure (hence the terminology, in reference to Möbius-like actions, i.e. such that every element is topologically conjugate to an element of $\PSL(2,\R)$). Differentially Fuchsian and $L$-differentially Fuchsian actions are spectrally Möbius-like. It is also quite straightforward to see that area-preserving actions are spectrally Möbius-like (see Proposition \ref{hyperbolic}).\\
\indent One can also define the spectrum $S(\rho):\Gamma \to \R^2$ as the data of the derivatives at fixed points for all elements of $\Gamma$.

\subsection{The case of a single diffeomorphism}

\indent The problem of knowing when a diffeomorphism that is topologically conjugate to a rotation is differentially conjugate to this rotation has been deeply studied. A well known theorem of Herman (\cite{Herman}) states that a differentiable conjugacy always exists provided the diffeomorphism has its  rotation number in a certain set of full Lebesgue measure (more precisely, if it satisfies a Diophantine condition, see \cite{Y84} for an exact description), but there are smooth examples where a differentiable conjugacy does not exist. In the area-preserving case, we do not have different behaviors:

\begin{prop} \label{elliptic} Let $f\in \emph{\Diff}(\Ss^1)$ be a fixed point free diffeomorphism If $f$ is area-preserving, then it is differentially conjugate to a rotation. \end{prop}

This result does not extend to diffeomorphisms with fixed points: there are some area-preserving circle diffeomorphisms that are not differentially conjugate to an element of $\PSL(2,\R)$. The following result treats the case corresponding to hyperbolic elements of $\PSL(2,\R)$.

\begin{prop} \label{hyperbolic} Let $f\in\Diff(\Ss^1)$ have exactly two fixed points $N$ and $S$. It is area-preserving if and only if it is spectrally Möbius-like \end{prop}

For parabolic diffeomorphisms (i.e. having one fixed point), the situation is more complicated. We will see there are some area-preserving examples that are not differentially conjugate to elements of $\PSL(2,\R)$, but that some diffeomorphisms with one fixed point do not preserve any volume form on the cylinder $\mathcal C$.

\subsection{The analytic case}

The counter examples produced by Proposition \ref{hyperbolic} never give an analytic volume form. Indeed, it appears that the analytic case is rigid. \\
\indent We say that   $\rho:\Gamma \to \Diff^\omega(\Ss^1)$ is \textbf{analytically Fuchsian} if there is a real analytic diffeomorphism  $h\in \Diff^\omega(\Ss^1)$ such that $h^{-1}\rho(\Gamma)h\subset \PSL(2,\R)$.

\begin{theo} \label{analytic_hyperbolic} Let $f\in \Diff^{\omega}(\Ss^1)$ have exactly two fixed points. If $f$ preserves an analytic volume form on $\mathcal C$, then $f$ is analytically conjugate to a hyperbolic element of $\PSL(2,\R)$. \end{theo}

For parabolic diffeomorphisms, there are some straightforward analytic counter examples. However, for non elementary representations, i.e. without any finite orbit on $\Ss^1$, there is also a rigidity phenomenon: 

\begin{theo} \label{analytic_non_elementary} If $\rho: \Gamma \to \emph{\Diff}^{\omega}(\Ss^1)$ is a non elementary representation preserving an analytic volume form on $\mathcal C$, then $\rho$ is analytically Fuchsian. \end{theo}

The treatment of the non elementary case will be  very different from the case of  a single diffeomorphism, mainly since the preserved volume form is unique for an analytic non elementary group.

\subsection{The topologically transitive case}

A theorem of Ghys, proved in \cite{Gh93},  states that any representation of a surface group (i.e. the fundamental group of a compact surface without boundary) into $\Diff(\Ss^1)$ with maximal Euler number is differentially Fuchsian. 
One particularity of these representations is that they are topologically transitive (they are even minimal: all orbits are dense). Given the condition of preserving a volume on $\mathcal C$, we also obtain a rigidity result.

\begin{theo} \label{transitive} Let $\rho: \Gamma \to \emph{\Diff}(\Ss^1)$ be a topologically transitive representation that preserves a $C^2$ volume form on $\mathcal C$. Then $\rho$ is differentially Fuchsian.  \end{theo}

\begin{rem} This result actually contains Proposition \ref{elliptic}, since diffeomorphisms that are topologically conjugate to a rational rotation are automatically differentially conjugate to this rotation, and irrational rotation are topologically transitive. \end{rem}
The $C^2$ regularity hypothesis is not only practical for the proof (it is linked to a notion of curvature), but it is important as there are some counter examples if we do not ask for enough regularity on the volume form.

\subsection{The exceptional minimal set case} The case of a single diffeomorphism suggests that the preservation of a volume form on $\mathcal C$ can be understood by looking at the fixed points. In the setting of Theorem \ref{transitive}, fixed points (when they exist) are dense in $\Ss^1$. We will now study groups for which the closure of fixed points is a Cantor set.
\subsubsection{Differential structure on the Cantor set}


\indent The definition of $L$-differential conjugacy suggests that we define a notion of diffeomorphisms between Cantor sets.\\
\indent If $C\subset \Ss^1$ is a closed set, then a function $f:C\to \Ss^1$ is $C^k$ in the Whitney sense if $f$ admits a Taylor development of order $k$ at every point of $C$, the coefficients being continuous functions. This is equivalent to asking that $f$ is the restriction to $C$ of a $C^k$ function on $\Ss^1$.\\
\indent We say that $f:C_1\to C_2$ (where $C_1$ and $C_2$ are two Cantor sets in $\Ss^1$) is a $C^k$ diffeomorphism if $f$ is a cyclic order preserving homeomorphism such that $f$ and $f^{-1}$ are $C^k$ in the Whitney sense. This is equivalent to asking that $f$ is the restriction to $C_1$ of a circle diffeomorphism.\\
\indent With this definition, we see that two non elementary representations $\rho_1,\rho_2:\Gamma \to \Diff(\Ss^1)$ are $L$-differentially conjugate if there is a homeomorphism $h\in \Homeo(\Ss^1)$ such that $h \rho_1 h^{-1} = \rho_2$ and such that the restriction $h_{/L_{\rho_1(\Gamma)}} : L_{\rho_1(\Gamma)}\to L_{\rho_2(\Gamma)}$ is a diffeomorphism.\\
\indent If $\rho:\Gamma \to \Diff(\Ss^1)$ is $L$-differentially Fuchsian, then let $h\in \Homeo(\Ss^1)$ be such that $\rho_0=h\rho h^{-1}$ is Fuchsian and such that $h_{/L_{\rho(\Gamma)}} : L_{\rho(\Gamma)}\to h(L_{\rho(\Gamma)})$ is a diffeomorphism. Let $\p\in \Diff(\Ss^1)$ be such that $\p_{/L_{\rho(\Gamma)}}=h_{/L_{\rho(\Gamma)}}$. We set $h_1 = \p \circ h^{-1}$ and $\rho_1 = h_1 \rho_0 h_1^{-1} = \p \rho \p^{-1}$. Since $\rho_1$ and $\rho$ are differentially conjugate, we see that $\rho$ is area-preserving if and only if $\rho_1$ is area-preserving. That way, we reduced the problem to a representation $\rho_1$ such that $\rho_1=h_1\rho_0 h_1^{-1}$ where $\rho_0$ is Fuchsian and $h_1$ is the identity on $L_{\rho_0(\Gamma)}$. We get a reformulation of Theorem \ref{area_preserving} which we will use for its proof.

\begin{theo} \label{non_Fuchsian} Let $\rho: \Gamma \to \emph{\PSL}(2,\R)$ be a convex cocompact representation and let $h\in \emph{\Homeo}(\Ss^1)$ be such that $h_{/L_{\rho(\Gamma)}}=Id$ and  $\rho_1 =h \rho h^{-1} $ has values in $\emph \Diff(\Ss^1)$. Then $\rho_1$ preserves a $C^2$ volume form on $\mathcal C$.\end{theo}

We will also show that some  specific  deformations of Schottky groups provide non differentially Fuchsian representations that satisfy the hypothesis of this theorem. The proof of Theorem \ref{non_Fuchsian} will take a substantial part of this paper (sections \ref{sec:flow} and \ref{sec:non_fuchsian}). Because of the lower regularity examples in the topologically transitive case mentioned above, it will be necessary to pay  particular attention to the regularity of the obtained volume form. \\
\indent A natural development would be to ask wether the converse is true.

\begin{conj} \label{conj_L_diff} If  $\rho: \Gamma \to \emph\Diff(\Ss^1)$ is non elementary and area-preserving, is it $L$-differentially Fuchsian? \end{conj}

\subsubsection{Infinitesimal rigidity}

Even though we do not have an answer to this exact question, we will see that there is some rigidity on the limit set by observing order three derivatives. The Schwarzian derivative, defined by $S(f)=(\frac{f'''}{f'}-\frac{3}{2}(\frac{f''}{f'})^2)dx^2$, is a quadratic differential that vanishes only for $f\in \PSL(2,\R)$. We obtain the following:

\begin{theo} \label{rigidity_schwarzian} If $\rho: \Gamma \to \Diff(\Ss^1)$ is a non elementary representation that preserves a smooth volume form on $\mathcal C$, then there is $h\in \Diff(\Ss^1)$ such that $S(h\circ\rho(\gamma)\circ h^{-1})(x)=0$ for all $\gamma \in \Gamma$ and $x\in L_{h\rho(\Gamma) h^{-1}}$. \end{theo}

\subsubsection{Spectrally Möbius-like groups}
In the case of a single hyperbolic diffeomorphism, preserving a volume form on $\mathcal C$ is equivalent to a condition on the derivatives at the fixed points. We can ask ourselves if it is also the case for more complicated groups.\\
\indent So far, it seems that spectrally Möbius-like is the weakest of all the properties defined above. However, for a group generated by a hyperbolic diffeomorphism, it is equivalent to being area-preserving. A natural question is to ask wether it is true for all group actions.
\begin{conj}\label{conj_spectral} If $\rho:\Gamma \to \Diff(\Ss^1)$ is topologically Fuchsian and  spectrally Möbius-like, is it area-preserving? \end{conj}

Note that  even though they seem to be indicating different directions, there is no obvious contradiction between this  statement and Question \ref{conj_L_diff} (i.e. we can ask wether spectrally Möbius-like actions are $L$-differentially Fuchsian).\\
\indent We will see that there is a positive answer to Question \ref{conj_spectral} for actions close to Fuchsian actions. For convenience, we will only treat the case of free groups.

\begin{theo} \label{spectrally} Let $\rho_0: \mathbb F_n \to \PSL(2,\R)$ be a convex cocompact representation. If $\rho_1: \mathbb F_n \to \Diff(\Ss^1)$ is sufficiently $C^1$-close to $\rho_0$, and if $\rho_1$ is spectrally Möbius-like, then $\rho_1$ is area-preserving. \end{theo}

Note that the hypothesis that $\rho_0$ is Fuchsian could be weakened by asking for $\rho_0$ to be $L$-differentially Fuchsian.\\
\indent For representations of surfaces groups, a theorem of Ghys in \cite{Gh92} (which preceded the result mentioned above) states that given $\rho_0 : \Gamma_g \to \PSL(2,\R)$ defined by a hyperbolic metric on the surface of genus $g$, any $C^1$-close representation $\rho_1 :\Gamma_g\to \Diff(\Ss^1)$ is differentially Fuchsian (notice that this does not mean that $\rho_1$ is differentially conjugate to $\rho_0$, but to another Fuchsian representation). In our context, we could ask if a representation $\rho_1:\Gamma \to \Diff(\Ss^1)$ that is spectrally Möbius-like and $C^1$-close to a convex cocompact representation $\rho_0\to \PSL(2,\R)$ is $L$-differentially Fuchsian. As in the case of surface groups, this does not mean that the existing topological conjugacy is a diffeomorphism between the limit sets. For this to be true, elements should have the same derivatives at their fixed points.\\
\indent Similarly, given $\rho_0,\rho_1:\Gamma \to \Diff(\Ss^1)$  such that $\rho_0$ is Fuchsian and that are topologically conjugate, if we assume that $\rho_0$ and $\rho_1$ have the same spectrum, are $\rho_1$ and $\rho_0$ $L$-differentially conjugate? In the context of hyperbolic dynamics, this is linked to understanding differentiable conjugacy by looking at the periodic data, i.e. the eigenvalues of the derivatives at periodic points (for Anosov diffeomorphisms of surfaces, the periodic date defines the system up to smooth conjugacy, see \cite{LMM88} and \cite{dlL}).

\subsection{Structure of the paper} We will start by studying topological conjugacy, then treat the elementary case (i.e. a single diffeomorphism). In section \ref{sec:tools}, we will introduce tools for the study of the non elementary case, mainly a notion of curvature associated to a smooth volume form on $\mathcal C$. The rigidity results concerning the non-elementary case, i.e. Theorem \ref{transitive}, Theorem \ref{analytic_non_elementary} and Theorem \ref{rigidity_schwarzian}, will be proved in section \ref{sec:rigidity}. Finally, we will prove Theorem \ref{non_Fuchsian} in sections \ref{sec:flow} and \ref{sec:non_fuchsian}, and Theorem \ref{spectrally} in section \ref{sec:spectrally}.

\section{Topological conjugacy} \label{sec:topological}

We deal with an action of a group $\Gamma$ on $\Ss^1$ and we wish to understand when it can preserve a measure on $\mathcal C=\Ss^1\times \Ss^1\setminus \Delta$. A result of Navas (Proposition 1.1 in  \cite{Navas_red}) states that for a certain type of measure, the action is topologically Fuchsian.

\begin{theo}[Navas] Let $\mu$ be a measure on $\mathcal C$ that is finite on compact sets, such that horizontal and vertical lines are negligible and  such that $\mu([a,b[\times ]b,c])=\infty$ for $a<b<c<a$ in $\Ss^1$. The group $\Gamma_{\mu}$ of circle homeomorphisms that preserve $\mu$ is topologically Fuchsian. \end{theo}

Navas used this result in \cite{Navas} to show that infinite Kazhdan groups cannot act on the circle by $C^2$ diffeomorphisms. Theorem \ref{topologic} deals with measures that are absolutely continuous with respect to the Lebesgue measure with a continuous density. If $\omega$ is a volume form on $\mathcal C$, then we will denote by $\Gamma_{\omega}$ the group of circle homeomorphisms $f$ such that the map $(x,y)\mapsto (f(x),f(y))$ of $\mathcal C$ preserves the measure defined by $\omega$.\\
\indent  In order to prove Theorem \ref{topologic}, we have to show that $\Gamma_{\omega}$ is topologically Fuchsian when $\omega$ is continuous.

\begin{lemme} If $\omega$ is a continuous volume form, then $\Gamma_{\omega} \subset \emph\Diff(\Ss^1)$ \end{lemme}

\begin{proof} Since the map $(f,f)$ preserves a measure in the class of the Lebesgue measure on $\mathcal C$, it is absolutely continuous, and so is $f$ on $\Ss^1$. The derivative of $f$ satisfies the relation $\omega(f(x),f(y))f'(x)f'(y)=\omega(x,y)$ for almost every $x,y$, therefore $f'$ is continuous and $f$ is $C^1$. A bootstrap argument  shows that if $\omega$ is $C^k$ with $k\ge 0$, then $\Gamma_{\omega}\subset \Diff^{k+1}(\Ss^1)$.   \end{proof}

The fact that $\Gamma_\omega$ is a group of diffeomorphisms gives us a more practical definition:
$$\Gamma_\omega =\{ f\in \mathrm{Diff}(\Ss^1) \vert \forall x\ne y~ \omega(f(x),f(y))f'(x)f'(y)=\omega(x,y)\}$$

Finding a conjugacy between a topologically Fuchsian group $\Gamma \subset \Diff(\Ss^1)$ and a subgroup of $\PSL(2,\R)$ is a rather complicated exercise. But there is a characterization of topologically Fuchsian groups that does not require to find an explicit conjugacy.\\
\indent First, we define the set $\Theta_3(\Ss^1)$ of distinct triples:
$$\Theta_3(\Ss^1)=\{(x,y,z)\in (\Ss^1)^3 \vert x\ne y\ne z\ne x\}$$

\begin{defi}  A group $\Gamma \subset \Homeo(\Ss^1)$ is  a convergence group if the action on $\Gamma$ of the space of distinct triples $\Theta_3(\Ss^1)$ is proper (i.e. for all compact set $K\subset \Theta_3(\Ss^1)$, the set $\Gamma_K=\{ g\in \Gamma \vert g.K\cap K\ne \emptyset\}$ is relatively compact). \end{defi}

Note that the definition of the properness of an action depends on a topology on the group. Here, the two candidates are the topology of $\Homeo(\Ss^1)$ and the compact open topology of $\Homeo(\Theta_3(\Ss^1))$, which happen to be identical.\\
\indent There is another classical definition of convergence groups, based on the dynamics of sequences in $\Gamma$. Their equivalence is shown in \cite{bowditch}. The main result on convergence groups is the following, proved in \cite{Gabai} and \cite{CJ}.

\begin{theo} A convergence group $\Gamma \subset \emph\Homeo(\Ss^1)$ is topologically Fuchsian. \end{theo}

\begin{proof}[Proof of Theorem \ref{topologic}] Let $h$ be the Riemannian metric on $\Theta_3(\Ss^1)$ defined by:

$$h_{(x,y,z)} = \frac{\omega(x,y)\omega(x,z)}{\omega(y,z)} dx^2 + \frac{\omega(y,z)\omega(y,x)}{\omega(z,x)} dy^2 + \frac{\omega(z,x)\omega(z,y)}{\omega(x,y)} dz^2$$
\indent It is a Riemannian metric on $\Theta_3(\Ss^1)$ that is preserved by the action of $\Gamma_\omega$. This implies that this action is proper (it is a straightforward consequence of Ascoli's Theorem), therefore $\Gamma_\omega$ is a convergence group, and is topologically Fuchsian.

  \end{proof}

\section{The elementary case}

In this section, we study the problem of differentiable conjugacy for a single diffeomorphism preserving a volume form on $\mathcal C$. Because such an element is topologically conjugate to an element of $\PSL(2,\R)$, we know that if it fixes at least three points, then it is the identity (this could actually be proved directly, without using the result for any group preserving a volume form on $\mathcal C$). We will study separately diffeomorphisms with a different number of fixed points. This corresponds to the classification of elements in $\PSL(2,\R)$: elliptic (no fixed point), parabolic (one fixed point) or hyperbolic (two fixed points).

\subsection{The elliptic case}

\indent We first look at the elliptic case, i.e.  fixed point free diffeomorphisms. All elliptic elements of $\PSL(2,\R)$ are conjugate (in $\PSL(2,\R)$, hence in $\Diff(\Ss^1)$) to rotations. The problem of knowing when a diffeomorphism topologically conjugate to a rotation is differentially conjugate to it has been studied deeply. There are examples for which a smooth conjugacy does not exist (including some irrational rotation numbers), however Herman proved that a smooth conjugacy exists when the rotation number lies in a set of full Lebesgue measure (\cite{Herman} discusses the general problem of differential conjugacy with a rotation). Luckily for us, the volume preserving case is much more simple.\\

\noindent \textbf{Proposition \ref{elliptic}.} \emph{Let $\p$ be a fixed point free diffeomorphism of $\Ss^1$. If it preserves a $C^k$ volume form on $\mathcal C$, then it is $C^{k+1}$ conjugate to a rotation. }

\begin{proof} Let $\omega$ be a volume form on $\mathcal C$ preserved by $\p$. We can define a Riemannian metric on $\Ss^1$ by $\Vert h\Vert_x^2 = \omega(x,\p(x)) \p'(x) h^2$. It is preserved by $\p$, therefore $\p$   is differentially conjugate to a rotation (because all  $C^k$ Riemannian metrics on the circle are $C^{k+1}$ homothetic to the euclidian metric whose isometries are rotations).
\end{proof}

Note that the Riemannian metric that we used can be seen as the restriction of the Lorentzian metric $\omega(x,y)dxdy$ on $\mathcal C$ to the graph of $\p$.

\subsection{The parabolic case}

We now deal with a diffeomorphism $\p$ that has exactly one fixed point $x_0\in \Ss^1$. Unlike the elliptic case, we will see that there is no rigidity. We can start by observing that the proof of the elliptic case does not apply here: the graph of $\p$ is not included in $\mathcal C$, therefore the  Riemannian metric that we used is only defined on $\Ss^1\setminus\{x_0\}$ and it only gives a conjugacy on $\Ss^1\setminus\{x_0\}$ with a translation of the real line, which only extends to a continuous conjugacy on  $\Ss^1$ with a parabolic element of $\PSL(2,\R)$, but this conjugacy is (in general) not smooth.\\
\indent There are immediate counter examples to differential conjugacy: we can  consider the family of  diffeomorphisms  $\p(x)=x(1+x^n)^{-\frac{1}{n}}$  (for $n$ odd) of $\R\proj^1=\R\cup \{\infty\}$. A preserved volume form   is given by $\vert x^n-y^n\vert^{-1-\frac{1}{n}}\,dx\wedge dy $. For $n\ne 1$, these diffeomorphisms are not differentially conjugate to an element of $\PSL(2,\R)$.\\
\indent However, all diffeomorphisms with one fixed point do not preserve a volume form on $\mathcal C$.

\begin{prop} \label{does_not_preserve} We see $\Ss^1$ as $\R\cup \{\infty\}$. Let $f\in \Diff(\Ss^1)$ be such that:
\begin{enumerate} \item $\emph{\textrm{Fix}}(f)=\{0\}$ \item $\forall x\in \intof{0}{1} ~ f(x)=(\emph{\textrm{Log}}(1+e^{\,x^{-2}}))^{-\frac{1}{2}}$ \item $\forall x\in \intfo{-1}{0} ~ f(x)=-(\emph{\textrm{Log}} (1+e^{\,x^{-4}}))^{-\frac{1}{4}}$ \end{enumerate} Then $f$ does not preserve any continuous volume form on $\mathcal C$.  \end{prop}

\begin{proof} Start by considering sequences $x_n \in \intof{0}{1}$ and $y_n\in \intfo{-1}{0}$ such that $x_n\to x\ne 0$ and $v_n=f^n(y_n)\to v\in \intfo{-1}{0}$ (this implies that $f^n(x_n)\to 0$ and $y_n\to 0$).\\
\indent If $f$ preserves a volume form $\omega$ on $\mathcal C$, then we find: \begin{equation} (f^n)'(x_n)(f^n)'(y_n)=\frac{\omega(x_n,y_n)}{\omega(f^n(x_n),f^n(y_n))}\to \frac{\omega(x,0)}{\omega(0,v)}\in \intoo{0}{+\infty}\tag{$*$} \label{preserves}  \end{equation}
\indent By rewriting $(f^n)'(y_n)=1/(f^{-n})'(v_n)$, we see that  computing the product $(f^n)'(x_n)(f^n)'(y_n)$ only uses  $f$ on $\intff{-1}{1}$.\\
\indent For $x\in \intof{0}{1}$, we find $f^n(x)=(\textrm{Log}(n+e^{\,x^{-2}}))^{-\frac{1}{2}}$ for all $n>0$, which gives: $$(f^n)'(x)=\frac{1}{x^3}\frac{1}{1+ne^{-x^{-2}}}(\textrm{Log}(n+e^{\,x^{-2}}))^{-\frac{3}{2}}$$ Similarly, for $y\in \intfo{-1}{0}$, we find $f^{-n}(y)=- (\textrm{Log}(n+e^{\,y^{-4}}))^{-\frac{1}{4}}$ and $$(f^{-n})'(y)=\frac{-1}{y^5}\frac{1}{1+ne^{-y^{-4}}}(\textrm{Log}(n+e^{\,y^{-4}}))^{-\frac{5}{4}}$$
This shows that: $$(f^n)'(x_n)(f^n)'(y_n)=\frac{(f^n)'(x_n)}{(f^{-n})'(v_n)}\sim \frac{-v^5}{x^3}e^{\,x^{-2}-v^{-4}} (\textrm{Log}(n))^{-\frac{1}{4}} \to 0$$ This is in contradiction with \eqref{preserves}.
\end{proof}

We will not try to give a necessary and sufficient condition for a diffeomorphism with one fixed point to preserve a volume form on $\mathcal C$. Note that the example in  Proposition \ref{does_not_preserve} is $C^{\infty}$-tangent to the identity at its fixed point. The same calculations could  give a smooth preserved volume form for a  diffeomorphism that is not infinitely tangent to the identity, as well as for some examples that are infinitely tangent to the identity. It seems that the key for preserving a volume form on $\mathcal C$ is having the same behavior on each side of the fixed point. 

\subsection{The hyperbolic case} \label{subsec:hyperbolic}

In the hyperbolic case (i.e. a diffeomorphism with two fixed points), we can start by seeing that all north/south diffeomorphisms cannot preserve a smooth volume.

\begin{lemme} \label{necessary} Let $f\in \Diff(\Ss^1)$ have exactly two fixed points $N$ and $S$. If $f$ is volume preserving, then $f'(N)\ne 1$ and $f'(N)f'(S)=1$. \end{lemme}

\begin{proof}Let $\omega$ be a continuous volume form on $\mathcal C$ preserved by $f$.  The identity $\omega(f(x),f(y))f'(x)f'(y)=\omega(x,y)$ considered  at the point $(N,S)\in \mathcal C$ shows that $f'(N)f'(S)=1$. Assume that  $f'(N)=1$ (hence $f'(S)=1$).\\
\indent Let $x(t)$ be a maximal solution of the Cauchy problem:
\begin{equation*}
\left\{
\begin{array}{cc}
  x'(t)  = & \frac{1}{\omega(x(t),S)}   \\
 x(0) = & N   \\
\end{array}
\right.
\end{equation*}
 \indent Not only does $x$ exist (Cauchy-Peano Theorem), but it is also unique (so are solutions to all equations $y'=F(y)$ in $\R$ where $F>0$).  Since $x'>0$, it is a diffeomorphism from an open interval $I\subset \R$ onto its image $J\subset \Ss^1\setminus \{S\}$. Let $\alpha = x^{-1} \circ f \circ x$. A simple calculation shows that $\alpha'(t)=1$ for all $t\in I$. Since $\alpha(0)=0$, we see that $\alpha =Id$ and $f(x)=x$ for all $x\in I$. Therefore the set of points $x\in \Ss^1\setminus \{S\}$ such that $f(x)=x$ and $f'(x)=1$ is open. It is also closed, and $\Ss^1\setminus \{S\}$ is connected, so $f=Id$.
\end{proof}

This property is  satisfied by a hyperbolic element of $\PSL(2,\R)$ (the derivatives at the fixed points are the squares of the eigenvalues of the matrix), and therefore by any diffeomorphism that is differentially conjugate to a hyperbolic element of $\PSL(2,\R)$, but there are examples of diffeomorphisms satisfying this property that have no differential conjugate in $\PSL(2,\R)$.\\
\indent Indeed, start with $\gamma\in \PSL(2,\R)$ a hyperbolic element. Let $N$ and $S$ be its fixed points. Let $\p\in\Homeo(\Ss^1)$ be such that: \begin{itemize} \item $\p$ fixes $N$ and $S$ \item $\p$ is a diffeomorphism on $\Ss^1\setminus\{S\}$ \item $\p$ is the identity in a neighborhood of $N$ \item $\p$ commutes with $\gamma$ in a neighborhood of $S$ \end{itemize}

Set $f=\p^{-1}\gamma\p\in \Diff(\Ss^1)$. If $f$ were differentially conjugate to an element of $\PSL(2,\R)$, then this element could be chosen to be $\gamma$. If $h^{-1}fh=\gamma$, then $\p\circ h$ is a diffeomorphism of $\Ss^1\setminus \{S\}$ that commutes with $\gamma$. This implies that  there is some $t\in \R$ such that $\p\circ h=\gamma_t$ on $\Ss^1\setminus \{S\}$ where $\gamma_s$ is the one parameter subgroup of $\PSL(2,\R)$ generated by $\gamma$. Indeed, in projective charts, we can see $\p\circ h$ as a diffeomorphism that commutes with a non trivial homothety $x\mapsto \lambda x$. The derivative is a continuous function on $\R$ invariant under $x\mapsto \lambda x$, hence constant, and $\p \circ h$ fixes $0$, hence is equal to some $x\mapsto \mu x$ in projective charts.   \\
\indent By continuity, the equality $\p\circ h=\gamma_t$ holds on all $\Ss^1$, and $\p$ is differentiable. Hence, if we choose $\p$ non differentiable, then $f$ is not differentially conjugate to an element of $\PSL(2,\R)$.\\
\indent The obstruction for a diffeomorphism with two fixed points to be differentially conjugate to an element of $\PSL(2,\R)$ is encoded in an element of $\Diff(\Ss^1)/\PSL(2,\R)$ called the Mather invariant (see \cite{Y95} for more details).\\
\indent Knowing this, the following result shows that preserving a volume form on $\mathcal C$ is not enough in order to be differentially conjugate to a homography.\\

\noindent\textbf{Proposition \ref{hyperbolic}.} \emph{Let $f\in\Diff^{k+1}(\Ss^1)$ ($k\ge 0$) have exactly two fixed points $N$ and $S$. It preserves a $C^k$ volume form on $\mathcal C$ if and only if $f'(N)f'(S)=1$ and $f'(N)\ne 1$.}
\begin{proof} Let $\lambda=f'(N)$ and let  $h_N: \Ss^1\setminus\{ S\} \to \R$ and $h_S: \Ss^1\setminus\{ N\} \to \R$ be the linearizations of $f$ at $N$ and $S$ (i.e. $h_N\circ f\circ h_N^{-1}(x)=\lambda x$ and $h_S\circ f\circ h_S^{-1}(x)=\lambda^{-1}x$). Let $U_1$ (resp. $U_2$) be a neighborhood of $(N,S)$ (resp. $(S,N)$) in $\mathcal C$ delimited by graphs of maps that commute with $f$ (hence invariant by $f$). The linearizations give us invariant volume forms (e.g.  $dx\wedge dy$ in coordinates) on $U_1$ and $U_2$. Since the action of $f$ on the complement of $U_1\cup U_2$ is proper (it is differentially conjugate to a  translation on the plane), we can find a smooth invariant volume form on $\mathcal C$ that coincides on $U_1$ and $U_2$ with the ones chosen above.
\end{proof}

\subsection{Analytic conjugacy}

In the fixed point free case, the conjugacy obtained is analytic when the diffeomorphism and the volume form are analytic. The previous construction in the hyperbolic case can never give a real analytic metric (given that the diffeomorphism is real analytic). In order to see this, we will introduce the Lorentz metric associated to a volume form on $\mathcal C$, which will give us  a notion of curvature. In the previous construction, the curvature is constant in a neighborhood of the axes, therefore any analytic prolongation to the whole cylinder would have constant curvature and the isometry group (that contains the diffeomorphism $f$) would be analytically Fuchsian.\\
\indent We can associate to the volume form $\omega(x,y)dx\wedge dy$ on $\mathcal C$ the Lorentz metric $g=\omega(x,y)dxdy$. If $\omega$ is $C^k$ with $k\ge 2$, then it defines the curvature as a real valued function $K$ on $\mathcal C$ that is $C^{k-2}$ (it is analytic when $\omega$ is analytic). The isometries of $g$ are the diagonal actions of circle diffeomorphisms that preserve $\omega$. In \cite{lorentz_surfaces}, we adopt the Lorentzian point of view and give a generalization of Theorem \ref{topologic} to a wider category of Lorentz surfaces. \\ 
\indent Lorentzian metrics, as well as Riemannian metrics, are examples of rigid geometric structures. We will use the fact that for an analytic rigid geometric structure, local vector fields generating isometries can be extended.\\

\noindent \textbf{Theorem \ref{analytic_hyperbolic}.} \emph{Let $f$ be an analytic diffeomorphism of $\Ss^1$ with exactly two fixed points. If it preserves an analytic volume form on $\mathcal C$, then it is analytically conjugate to an element of $\PSL(2,\R)$.}

\begin{proof} Let $\omega$ be an analytic volume form preserved by $f$. By Lemma \ref{necessary}, if $N$ and $S$ are the fixed points of $f$, then $\lambda=f'(N)\ne 1$ and $f'(S)=\lambda^{-1}$. By considering the linearizations of $f$ around its fixed points, we see that the diagonal action of $f$ is analytically conjugate in a neighborhood of $(N,S)$ to the map $(x,y)\mapsto (\lambda x,\lambda^{-1}y)$ in a neighborhood of $(0,0)$. Since it preserves the volume form $dx\wedge dy$ in those coordinates, we can write $\omega = e^{\sigma} dx\wedge dy$ in coordinates where $\sigma$ is an analytic function that satisfies $\sigma(\lambda x, \lambda^{-1}y)=\sigma(x,y)$. By writing $\sigma$ in its power series around $(0,0)$ and considering the invariance equation, we see that all the terms in $x^ny^p$ with $n\ne p$ must have zero as their  coefficient, therefore we can write $\sigma = f(xy)$ where $f$ is an analytic function, and the form $\omega$ is preserved (around the fixed point $(N,S)$) by the one parameter group associated to $f$.\\
\indent We will now apply the main result of \cite{Amores}: a local Killing field (i.e. a vector field that generates a flow of isometries) on a simply connected real analytic Lorentz manifold admits a unique extension to the whole manifold (the paper treats the more general case of finite type $G$-structures, which includes Lorentz metrics).\\
\indent In order to apply this result, consider a map from $\intff{N}{S}$ to  $\intff{S}{N}$ that commutes with the (topological) one parameter group associated to $f$, and let $U$ be the complement of the graph of this map. It is simply connected open set of $\mathcal C$ that is invariant under  the one parameter group associated to $f$ and that contains $(N,S)$ and $(S,N)$.  There is a vector field $\mathfrak X$ on $U$ that preserves $\omega$ and such that the time one map is $(f,f)$. Since the vector field $\mathfrak X$ has the form $\mathfrak X(x,y)=(\mathfrak x(x), \mathfrak x(y))$ where $\mathfrak x$ is defined on all $\Ss^1$, it is complete, and the map $f$ is the time $1$ of the flow of the analytic vector field $\mathfrak x$, hence $f$ is analytically conjugate to an element of $\PSL(2,\R)$ (the Mather invariant of the time one map of a flow is trivial, see \cite{Y95}).

\end{proof}

However, there are non Fuchsian examples in the parabolic case. Indeed, for $n\in \N$  odd and greater than $1$, consider the examples $f(x)=x(1+x^n)^{-1/n}$ discussed in the differentiable case. It is analytic on $\R\proj^1=\R\cup\{\infty\}$ (because $\frac{1}{f}$ is analytic in a neighbourhood of $-1$). It preserves the volume form $ \vert x^n-y^n\vert^{-1-1/n}dx \wedge dy$ which extends analytically to $\Ss^1\times \Ss^1\setminus \Delta$.\\
\indent The example of a parabolic diffeomorphism that does not preserve a volume form given in Proposition \ref{does_not_preserve} is not analytic. We suspect that in the parabolic case, all analytic diffeomorphisms preserve an analytic volume form on $\mathcal C$.

\section{Tools for the non elementary case} \label{sec:tools}

\subsection{The limit set}

Given a group $\Gamma \subset \Homeo(\Ss^1)$, then exactly one of the following conditions is satisfied (see \cite{Gh01} for a proof and more detail):
\begin{enumerate} \item $\Gamma$ has a finite orbit \item All orbits of $\Gamma$ are dense \item There is a compact $\Gamma$-invariant subset $K\subset \Ss^1$ which is infinite and different from $\Ss^1$, such that the orbits of points of $K$ are dense in $K$. \end{enumerate}

\indent In the third case, the set $K$ is unique, and it is homeomorphic to a Cantor set. It is called an \textbf{exceptional minimal set}. We can call a group $\Gamma\subset \Homeo(\Ss^1)$ \textbf{non elementary} if it does not have any finite orbit (this definition is not standard since we usually want to call the group generated by an irrational rotation elementary), and use $L_{\Gamma}$ to denote $\Ss^1$ in the second case and  the $\Gamma$-invariant compact set $K$ in the third case (we call $L_\Gamma$ the \textbf{limit set} of $\Gamma$).\\
\indent If $\Gamma \subset \PSL(2,\R)$ is non elementary and possesses hyperbolic elements (to avoid the case mentioned above), then $L_\Gamma$ is the intersection of the circle at infinity $\partial_\infty \h^2$ with the closure of the orbit $\overline{\Gamma.x}$ in $\overline{\h^2}$, independently of the point $x\in \h^2$.


\subsection{Projective structures and curvature}

One of the advantages of considering the Lorentz metric associated to a volume form on $\mathcal C$ is that it gives us a notion of curvature. In the two dimensional case, it is a function $K:\mathcal C\to \R$ that is $C^{k-2}$ when the volume form is $C^k$. In our setting, it has a simple expression:
$$K=\frac{2}{\omega}\frac{\partial^2 \Log\omega}{\partial x \partial y}$$
\indent  It is invariant under the diagonal actions of circle diffeomorphisms that preserve the volume form (because they are isometries). This will give an important subset of $\mathcal C$ on which the curvature is constant.

\begin{lemme} \label{curvature_limit_set} Let $\rho:\Gamma \to  \Diff(\Ss^1)$ be a   representation that preserves a $C^2$ volume form on $\mathcal C$. Assume that there is a least one hyperbolic element. Then the curvature $K$ is constant on $(L_{\rho(\Gamma)}\times \Ss^1\cup \Ss^1\times L_{\rho(\Gamma)})\setminus \Delta$. \end{lemme}

\begin{proof} Let $\omega$ be such a volume form.  If $\gamma \in \Gamma$ and $\rho(\gamma)$ has two fixed points $N,S$ in $\Ss^1$, then we can consider the fixed point $p=(N,S)\in \mathcal C$. The orbits of points of the axes $\{N\}\times \Ss^1\setminus \{N\}$ and $\Ss^1\setminus\{S\} \times \{S\}$ accumulate on $p$, therefore the curvature at these points have the same value $K( p)$. Given two hyperbolic elements of $\Gamma$, the axes meet, therefore the curvature has the same value on the axes of all hyperbolic elements of $\Gamma$. Since a fixed point of a hyperbolic element has a dense orbit in $L_{\rho(\Gamma)}$, we find that $K$ is constant on  $(L_{\rho(\Gamma)}\times \Ss^1\cup \Ss^1\times L_{\rho(\Gamma)})\setminus \Delta$ \end{proof}

Note that the exact same proof works for any continuous function on $\mathcal C$ invariant under the action of $\Gamma$. The specificity of the curvature is that when it is constant, the metric is locally isometric to a model space. We will now see how this can give a global conjugacy for the isometry group. It is in general more difficult to have global results on constant curvature Lorentz manifolds than on Riemannian manifolds, because the associated $(G,X)$-structure  is not always complete (the developing map may not be a covering map, whereas it is always the case for Riemannian isometries).\\

\indent Another tool that we get with a Lorentz metric is geodesics. Horizontal and vertical lines in $\mathcal C=\Ss^1\times \Ss^1\setminus \Delta$ are geodesics (because they are the only isotropic curves), which gives us some specific parametrisations. We will translate them in terms  of projective structures on one dimensional manifolds.\\
\indent A \textbf{projective structure} on a one-dimensional manifold $I$ is an atlas $(U_i,f_i)$  with $f_i:U_i\to \R\proj^1$ such that the transition maps $f_i\circ f_j^{-1}$ are projective diffeomorphisms (i.e. restrictions of elements of $\PSL(2,\R)$).  If $f$ is a diffeomorphism between two projective one-dimensional manifolds $I$ and $J$, then one can  define a quadratic differential $s(f)$ on $I$, called the \textbf{Schwarzian derivative} of $f$,  by $s(f)=\left(\frac{f'''}{f'}-\frac{3}{2}(\frac{f''}{f'})^2\right)dx^2$ in projective charts. Then $f$ is a projective diffeomorphism (i.e. $f$ has the form $x\mapsto \frac{ax+b}{cx+d}$ in projective charts) if and only if $s(f)=0$ (see \cite{Gh93} for more details).\\
\indent Note that some links between the Schwarzian derivative and Lorentzian geometry have been studied, mostly concerning the geodesic curvature (see \cite{DO}).\\
\indent Geodesics inherit a projective structure, the charts being given by the different parametrisations of the geodesic (the coordinate changes are affine, therefore projective). Recall that the geodesic equations are the following: \begin{eqnarray*}  x''+\frac{1}{\omega}\frac{\partial \omega}{\partial x}x'^2 &=& 0\\ y''+\frac{1}{\omega}\frac{\partial \omega}{\partial y}y'^2 &=&0\end{eqnarray*}
\indent A representation $\rho: \Gamma \to \Diff(\Ss^1)$ is differentially Fuchsian if and only if it preserves a projective structure on $\Ss^1$ that is equivalent to the standard structure on $\R\proj^1$ (because a conjugacy between $\rho$ and a Fuchsian representation is the same as a projective diffeomorphism with $\R\proj^1$).  Therefore in order to show that a representation is differentially Fuchsian, we can proceed in two steps: first we find an invariant projective structure, then we show that it is equivalent to the standard projective structure on $\R\proj^1$. This is what we will use in the proof of the following result.

\begin{lemme} \label{constant_curvature} Let $\rho: \Gamma \to \Diff(\Ss^1)$ be a representation that preserves a $C^2$ volume form on $\mathcal C$. Assume that its curvature is constant. Then $\rho$ is differentially Fuchsian. \end{lemme}

\begin{proof} Given $y\in \Ss^1$, we consider a diffeomorphism $f_y:  \Ss^1\setminus\{y\} \to \R$ given by a parametrization of the horizontal circle $\Ss^1\setminus\{y\}\times \{y\}$ as a geodesic for the Lorentz metric associated to $\omega$. This gives us an atlas of $\Ss^1$, and we will first show that it is a projective structure, i.e. that  the transition maps $f_{y'}\circ f_y^{-1}$ are projective. For any sequence $y_1,\dots,y_n$, we can decompose $f_{y'}\circ f_y^{-1}$: 
$$f_{y'}\circ f_y^{-1}=(f_{y'}\circ f_{y_n}^{-1}) \circ (f_{y_n}\circ f_{y_{n-1}}^{-1})\circ \cdots \circ (f_{y_1}\circ f_{y}^{-1})$$ \indent Since the composition of projective maps is projective, it is enough to show that $f_{y'}\circ f_y^{-1}$ is projective when $y$ and $y'$ are sufficiently close.\\
\indent Given $(x,y)\in \mathcal C$, we can find a local isometry with the model space of constant curvature, which can also be seen (locally) as a volume form on $\mathcal C$ ($dx\wedge dy$ for zero curvature, $\pm \frac{4dx\wedge dy}{(x-y)^2}$ for curvature $\pm 1$). An isometry sends parametrized geodesics onto parametrized geodesics, hence $f_{y'}\circ f_y^{-1}$ is equal to the analogue in the model space, and it is projective because it is the case in the model space.\\
\indent Given an element $\gamma \in \Gamma$, we know that $f_y\circ \rho(\gamma)$ is also the inverse of the parametrization of a geodesic, hence $f_{y'}\circ\rho(\gamma)\circ f_y^{-1}$ is projective, and the projective structure that we defined is preserved by $\rho$.\\
\indent To conclude, we separate two cases. If there is an element of $\Gamma$ with a fixed point in $\Ss^1$, then Lemma 5.1  of \cite{Gh93} concludes that the projective structure is equivalent to the standard structure on $\R\proj^1$, and $\rho$ is differentially Fuchsian.\\
\indent If all elements are elliptic, then applying Theorem \ref{topologic} shows that $\rho$ is topologically conjugate to a representation in $\PSL(2,\R)$ with only elliptic elements, and it is therefore conjugate to a subgroup of $\textrm{SO}(2,\R)$ (see \S 7.39 in \cite{beardon}). In particular, it is abelian, and the same argument as in Proposition \ref{elliptic} ($\rho$ preserves the Riemannian metric $\omega(x,\rho(\gamma_0) x)\rho(\gamma_0)'(x)dx^2$ on $\Ss^1$ where $\gamma_0$ is any element in $\Gamma\setminus \{e\}$) shows that $\rho$ is differentially conjugate to a representation in $\textrm{SO}(2,\R)\subset \PSL(2,\R)$.

\end{proof}

One could ask if, more generally,  the projective structures given by the parametrisations of isotropic geodesics are invariant under the isometry group (which is the same as asking for the maps $f_{y'}\circ f_y^{-1}$ to be projective). In \cite{these} (Chapter 4, section 3), we compute the Schwarzian derivative of the maps $f_{y'}\circ f_y^{-1}$  and show that it  only vanishes when the curvature is constant. We will see in Proposition \ref{invariant_function} that it is not always the case.

\section{Rigidity results for non elementary groups} \label{sec:rigidity}

\subsection{Topologically transitive actions}
In the topologically transitive case, i.e. when the limit set is the whole circle, the situation is rigid (provided sufficient regularity). We will use the results stated above to show that the curvature is constant.\\

\noindent \textbf{Theorem \ref{transitive}.} \emph{Let $\rho:\Gamma \to  \Diff(\Ss^1)$ be a  topologically transitive representation that preserves a $C^2$ volume form on $\mathcal C$. Then $\rho$ is differentially Fuchsian.}

\begin{proof} If there is a hyperbolic element, then Lemma \ref{curvature_limit_set}  states that the curvature is constant on $\mathcal C$ and Lemma \ref{constant_curvature} allows us to conclude.\\
\indent We now treat the case where there are no hyperbolic elements, i.e. all elements are elliptic or parabolic. First assume that there is a parabolic element $\gamma$. Let $x_0\in \Ss^1$ be its fixed point.  If there is another parabolic element $\delta$ with a different fixed point, then either $\gamma \delta$ or $\gamma^{-1}\delta$ is hyperbolic, hence we can assume that all parabolic elements fix $x_0$. Since the group is not elementary, there is a non trivial elliptic element $\alpha$. The conjugate $\alpha \gamma \alpha^{-1}$ is a parabolic element whose fixed point is $\rho(\alpha)(x_0)\ne x_0$, and as we just showed this implies the existence of a hyperbolic element in $\Gamma$. We have shown that the existence of a parabolic element in a non elementary group preserving a volume form on $\mathcal C$ implies the existence of a hyperbolic element.\\
\indent We are left with the case where all elements are elliptic, where we simply notice that we did not use the fact that the curvature is constant in this case in the proof of Lemma \ref{constant_curvature}. \end{proof}

The regularity of the preserved volume form is  essential in this result. If $(S,h)$ is a smooth compact Riemannian surface of negative curvature, then the fundamental group $\pi_1(S)$ acts isometrically on the universal cover $\tilde S$, hence it acts on its boundary at infinity $\partial_{\infty}\tilde S \approx \Ss^1$. 
To find an invariant volume form, consider the space of oriented geodesics of $\tilde S$. It can be seen as $\rm{T}^1\tilde S/\R$ where the action of $\R$ is the geodesic flow, and $\pi_1(S)$ preserves the form $\omega=d\lambda$ where $\lambda$ is the projection of the Liouville $1$-form on $\rm{T}^1\tilde S$. An oriented geodesic is given by a starting point and an end point on $\partial_{\infty}\tilde S$, which gives an identification between $\rm{T}^1\tilde S/\R$ and $\mathcal C=\partial_{\infty}\tilde S\times \partial_{\infty}\tilde S \setminus \Delta$. This identification is only a $C^1$-diffeomorphism (its regularity is exactly the regularity of the weak stable and weak unstable foliations of the geodesic flow), so the volume form obtained on $\Ss^1\times \Ss^1\setminus \Delta$ is only continuous. A result of Ghys in \cite{Gh87} states that if the identification $\rm{T}^1\tilde S/\R\approx \mathcal C$ is $C^2$, then $(S,h)$ has constant curvature. \\
\indent It is not even clear whether the  regularity required in Theorem \ref{transitive} can be lowered to $C^{1,1}$ (i.e. $C^1$ with a Lipschitz derivative). In this case, the curvature is defined almost everywhere, and is locally $L^\infty$. To ensure that such a function is constant almost everywhere, the right notion is no longer topological transitivity but ergodicity. A group action by diffeomorphisms on a manifold is  \textbf{ergodic} if all invariant measurable sets are either negligible  or of full measure (for the class of the Lebesgue measure). If an action on the circle $\rho:\Gamma \to \Diff(\Ss^1)$ is ergodic, then the diagonal action on $\mathcal C$ also is. The question of knowing whether topologically transitive actions on the circle are ergodic is very important in the theory of circle diffeomorphisms. It has been proven to be true for analytic actions of finitely generated free groups (in \cite{Herman} for $\Z$, and  \cite{DKN09},\cite{DKN13} for $\mathbb F_n$, $n\ge 2$), and it is expected to be true for $C^2$ actions of finitely presented groups. This could be applied in our situation (if the metric is $C^{1,1}$, then isometries are $C^{2,1}$), but we would still have to prove that if the curvature is constant almost everywhere, then we have an isometry with the model space.

\subsection{Analytic rigidity}

As in the elementary case, analyticity also provides more rigidity in the non elementary case.\\

\noindent \textbf{Theorem \ref{analytic_non_elementary}.} \emph{Let $\rho: \Gamma \to \emph{\Diff}^{\omega}(\Ss^1)$ be a non elementary  representation that preserves an analytic volume form on $\mathcal C$. Then $\rho$ is analytically Fuchsian.}

\begin{proof}

 Applying Lemma \ref{curvature_limit_set} we see that the curvature is constant on the set $(L_{\rho(\Gamma)}\times \Ss^1\cup \Ss^1\times L_{\rho(\Gamma)})\setminus \Delta$. The  analyticity of the curvature implies that it is constant on $\mathcal C$ (consider the function along horizontal and vertical lines and the fact that $L_{\rho(\Gamma)}$ is without isolated points), and Lemma \ref{constant_curvature}  implies that $\rho$ is analytically Fuchsian.   \end{proof}

 \subsection{Exceptional minimal set and curvature}
We saw that the curvature is constant on $(L_{\rho(\Gamma)}\times \Ss^1\cup \Ss^1\times L_{\rho(\Gamma)})\setminus \Delta$, but we cannot have anything better than this. Indeed, we can construct metrics with non constant curvature that are preserved by non elementary Fuchsian groups. Since such a group $\Gamma$ preserves a volume form, any other preserved volume form is given by the product with an invariant function.

\begin{prop} \label{invariant_function} Let $ \Gamma \subset \PSL(2,\R)$ be a non elementary and non topologically transitive subgroup. Then there is a non constant smooth function $\sigma: \mathcal C \to \R$ that is $\Gamma$-invariant. \end{prop}

\begin{proof}  Start by writing $\Ss^1\setminus L_{\Gamma}= \bigcup_{i\in \N}I_i$ as the union of its connected components. We start by setting $\sigma=0$ on $(L_{\rho(\Gamma)}\times \Ss^1\cup \Ss^1\times L_{\rho(\Gamma)})\setminus \Delta$ and on $I_i\times I_i\setminus \Delta$ for $i\in \N$. For $x\in I_i\times I_j$ with $i\ne j$, consider $R_1,R_2,R_3,R_4$ the four rectangles that have $x$ as one corner and a corner of $I_i\times I_j$ as the opposite corner (see Figure \ref{fig:invariant_functions}). Let $\sigma(x)=\omega(R_1)\omega(R_2)\omega(R_3)\omega(R_4)$ where $\omega$ is the volume form $\frac{4dx\wedge dy}{(x-y)^2}$. By using the explicit formula $\omega(\intff{a}{b}\times\intff{c}{d})=4\Log([a,b,c,d])$ where $[a,b,c,d]=\frac{a-c}{a-d}\frac{b-d}{b-c}$ is the cross-ratio, we see that $\sigma$ is continuous. The function $\sigma$ is smooth in the interior of rectangles $I_i\times I_j$, i.e. where it is non zero. If $F:\R \to \R$ is smooth and constant on a neighbourhood of $0$ sufficiently small so that $F\circ \sigma$ is not constant, then $F\circ \sigma$ is  $\Gamma$-invariant, non constant and smooth.\\
\indent There are many other ways of constructing invariant functions. We could set $\sigma(x)$ on $I_i\times I_i$ to be $F(\omega(R ))$ where $R$ is the rectangle amongst $R_1,R_2,R_3$ and $R_4$ defined above that is included in $\Ss^1\times \Ss^1\setminus \Delta$ (see Figure \ref{fig:invariant_functions}).\\
\indent Finally, we could also choose $\sigma$ arbitrarily on the squares $I_i\times I_i$ where $i$ is in a fundamental domain for the action of $\Gamma$ on the connected components of $\Ss^1\setminus L_\Gamma$, and let $\sigma$ be constant on rectangles $I_i\times I_j$ with $i\ne j$.
\end{proof}

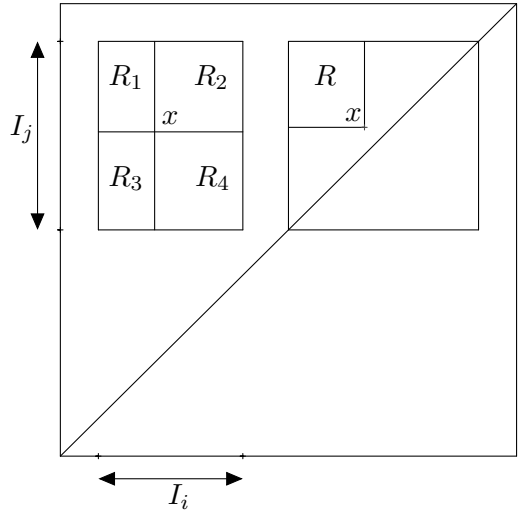
\begin{figure}[h]
\definecolor{uuuuuu}{rgb}{0.27,0.27,0.27}
\begin{tikzpicture}[line cap=round,line join=round,>=triangle 45,x=1.0cm,y=1.0cm]
\clip(-6.5,-1.9) rectangle (11.6,5.2);
\draw (-3,5)-- (3,5);
\draw (3,5)-- (3,-1);
\draw (3,-1)-- (-3,-1);
\draw (-3,-1)-- (-3,5);
\draw (-3,-1)-- (3,5);
\draw (-2.5,4.5)-- (-0.6,4.5);
\draw (-0.6,4.5)-- (-0.6,2);
\draw (-0.6,2)-- (-2.5,2);
\draw (-2.5,4.5)-- (-2.5,2);
\draw (-2.5,3.3)-- (-0.6,3.3);
\draw (-1.76,2)-- (-1.76,4.5);
\draw [<->] (-3.3,2) -- (-3.3,4.5);
\draw [<->] (-0.6,-1.3) -- (-2.5,-1.3);
\draw (-1.72,-1.24) node[anchor=north west] {$I_i$};
\draw (-3.8,3.66) node[anchor=north west] {$I_j$};
\draw (-2.5,4.32) node[anchor=north west] {$R_1$};
\draw (-1.38,4.32) node[anchor=north west] {$R_2$};
\draw (-2.5,2.98) node[anchor=north west] {$R_3$};
\draw (-1.36,2.98) node[anchor=north west] {$R_4$};
\draw (-1.8,3.7) node[anchor=north west] {$x$};
\draw (0,4.5)-- (0,2);
\draw (0,2)-- (2.5,2);
\draw (2.5,2)-- (2.5,4.5);
\draw (2.5,4.5)-- (0,4.5);
\draw (1,4.5)-- (1,3.36);
\draw (0,3.36)-- (1,3.36);
\draw (1.1,3.75) node[anchor=north east] {$x$};
\draw (0.2,4.32) node[anchor=north west] {$R$};
\begin{scriptsize}
\draw [color=black] (-2.5,-1)-- ++(-1.0pt,0 pt) -- ++(2.0pt,0 pt) ++(-1.0pt,-1.0pt) -- ++(0 pt,2.0pt);
\draw [color=black] (-0.6,-1)-- ++(-1.0pt,0 pt) -- ++(2.0pt,0 pt) ++(-1.0pt,-1.0pt) -- ++(0 pt,2.0pt);
\draw [color=black] (-3,2)-- ++(-1.0pt,0 pt) -- ++(2.0pt,0 pt) ++(-1.0pt,-1.0pt) -- ++(0 pt,2.0pt);
\draw [color=black] (-3,4.5)-- ++(-1.0pt,0 pt) -- ++(2.0pt,0 pt) ++(-1.0pt,-1.0pt) -- ++(0 pt,2.0pt);
\draw [color=uuuuuu] (1,3.36)-- ++(-1.0pt,0 pt) -- ++(2.0pt,0 pt) ++(-1.0pt,-1.0pt) -- ++(0 pt,2.0pt);
\end{scriptsize}
\end{tikzpicture}
\caption{Construction of invariant functions} \label{fig:invariant_functions}
\end{figure}

This result takes away all hope of finding a differential conjugacy with the invariance of the curvature (there are enough ways to produce an invariant function to ensure that there are preserved metrics with non constant curvature).

\subsection{Infinitesimal rigidity on the limit set} The question of differentiable conjugacy appears to be difficult and a way of dealing with a more simple problem is to linearize the conjugacy equation, i.e. considering the derivatives of the equations $\rho_1(\gamma)=h^{-1}\circ \rho_0(\gamma) \circ h$ where $\rho_1:\Gamma\to \Diff(\Ss^1) $ is the data and $h\in \Diff(\Ss^1)$ and $\rho_0:\Gamma\to \PSL(2,\R)$ are the unknowns. First and second order derivatives remain quite complicated, but the third order is more simple because elements of $\PSL(2,\R)$ can be defined as the solutions of  a third order differential equation. But since we know that it is not always possible to have a differentiable conjugacy on the whole circle (the proof will be exposed in sections \ref{sec:flow} and \ref{sec:non_fuchsian}), we can only look at subsets of the circle. In the counter example that we will construct, the conjugacy is differentiable along the limit set. This is interesting because the limit set is the subset of the circle that contains the non trivial dynamical behavior.\\
\indent We have  already seen that a volume form on $\mathcal C$ endows the horizontal and vertical lines with projective structures. We showed that in the constant curvature case, they give the same projective structure on $\Ss^1$. Before we give a statement of a result, we will  reformulate this.\\
\indent  We will denote by $E^1$ (resp. $E^2$) the sub-bundle of $T\mathcal C$ consisting of horizontal (resp. vertical) lines. If $p\in X$ and $u\in E^2( p)$, then $\alpha_u$ is the geodesic with initial condition $u$, and $\mathcal C^u_t$ is the horizontal circle passing through $\alpha_u(t)$. We will consider the holonomy map $H_t^u: \mathcal C^u_0 \to \mathcal C^u_t$ (which is defined everywhere on the circle except at two points, see Figure \ref{fig:holonomy}). The Schwarzian derivative $K_u(t)=S(H_t^u)$ relatively to the projective structure on $\mathcal C^u_t$ given by the Lorentzian metric is a field of quadratic forms on $E^1$, and we will mostly consider $k_u(t)=K_u(t)( p) \in S^2(E^1( p))$. Note that if $\rho$ were Fuchsian, then $k_u(t)$ would vanish everywhere (this is what we have shown in the constant curvature case). If it were $L$-differentially Fuchsian, then it would vanish when the base point of $u$ is in $L_{\Gamma}\times L_{\Gamma}$, therefore the following result can be interpreted as a rigidity result.

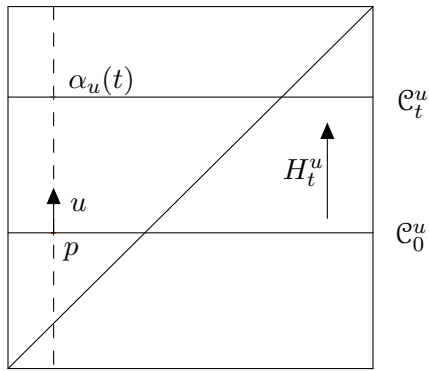
\begin{figure}[h]
\definecolor{uuuuuu}{rgb}{0.27,0.27,0.27}
\definecolor{ttqqqq}{rgb}{0.2,0,0}
\begin{tikzpicture}[line cap=round,line join=round,>=triangle 45,x=1.0cm,y=1.0cm,scale=0.6]
\clip(-8.5,-4.98) rectangle (16.9,6.3);
\draw (-2,5)-- (6,5);
\draw (6,5)-- (6,-3);
\draw (6,-3)-- (-2,-3);
\draw (-2,5)-- (-2,-3);
\draw (6,5)-- (-2,-3);
\draw (-2,0)-- (6,0);
\draw (-2,3)-- (6,3);
\draw [->] (-1,0) -- (-1,1);
\draw [dash pattern=on 5pt off 5pt] (-1,5)-- (-1,-3);
\draw [->] (5,0.32) -- (5,2.44);
\draw (-1,0) node[anchor=north west] {$p$};
\draw (-0.86,0.98) node[anchor=north west] {$u$};
\draw (-0.9,3.9) node[anchor=north west] {$\alpha_u(t)$};
\draw (3.8,1.86) node[anchor=north west] {$H^u_t$};
\draw (6.3,3.38) node[anchor=north west] {$\mathcal C^u_t$};
\draw (6.26,0.44) node[anchor=north west] {$\mathcal C^u_0$};
\begin{scriptsize}
\draw [color=ttqqqq] (-1,0)-- ++(-1.5pt,0 pt) -- ++(3.0pt,0 pt) ++(-1.5pt,-1.5pt) -- ++(0 pt,3.0pt);
\draw [color=uuuuuu] (-1,3)-- ++(-1.5pt,0 pt) -- ++(3.0pt,0 pt) ++(-1.5pt,-1.5pt) -- ++(0 pt,3.0pt);
\end{scriptsize}
\end{tikzpicture}
\caption{The holonomy map $H^u_t$}
\label{fig:holonomy}
\end{figure}

\begin{theo} If $\rho: \Gamma \to \Diff(\Ss^1)$ preserves a smooth volume form on $\mathcal C$, and if $\rho(\Gamma)$ is non elementary, then $k_u(t)=0$ for all $p\in L_{\rho(\Gamma)}\times L_{\rho(\Gamma)}\setminus \Delta$ and all $u\in E^2( p), t\in \R$. \end{theo}

\begin{proof} If $\gamma\in \Gamma$, then $H_t^{\gamma.u}=\gamma\circ H_t^u\circ \gamma^{-1}$. Since the group $\Gamma$ acts isometrically with respect to the Lorentz metric, it preserves the projective structures, and the cocycle relation on the Schwarzian derivative gives us $K_{\gamma.u}(t)=\gamma_*K_u(t)$.\\
\indent Let us now remark that since the space $S^2(E^1( p))$ is one-dimensional, we can write $k_u(t)(v)=F(u,t) < u,v>^2$ for all $v\in E^1( p)$ (where $<\cdot,\cdot>$ is the Lorentz metric associated to the preserved volume form). The relation $K_{\gamma.u}(t)=\gamma_*K_u(t)$ gives us $F(\gamma.u,t) =F(u,t)$.\\
\indent If $a>0$, then we have $\alpha_{au}(t)=\alpha_u(at)$, which gives us $K_{au}(t)=K_u(at)$.\\
\indent We will now study the case where $p$ is a fixed point of $\gamma$. We write $p=(x,y)$ and $\gamma'(x)=\lambda^{-1}$, $\gamma'(y)=\lambda$, with $\lambda \ne 1$. Since $\gamma.u=\lambda u$, we have $k_u(\lambda t)=k_{\lambda u}(t)=k_{\gamma.u}(t)=\gamma_*k_u(t)=\lambda^2k_u(t)$, which implies  that $F(u,\lambda t)=\lambda^2F(u,t)$, therefore (because of the differentiability of the map $t\mapsto F(u,t)$) there is a real number $c(u)$ such that $F(u,t)=c(u)t^2$. \\
\indent We now wish to extend this to $L_{\rho(\Gamma)}\times L_{\rho(\Gamma)}\setminus \Delta$. If we fix $t\in \R$ and $k> 2$, the function $\frac{\partial^k}{\partial t^k}F(u,t)$ is invariant under the action of $\Gamma$, and it is equal to $0$ on all vectors tangent to fixed points of $\Gamma$, therefore by continuity it is equal to $0$ on $L_{\rho(\Gamma)}\times L_{\rho(\Gamma)}\setminus \Delta$, i.e. $F(u,t)=a(u)+b(u)t+c(u)t^2$. Since the coefficients are continuous, we have $a(u)=b(u)=0$, i.e. $F(u,t)=c(u)t^2$.
\newline 
\indent We will finally compute $k_u(t+s)$ in two ways in order to conclude. We choose $p\in L_{\rho(\Gamma)}\times L_{\rho(\Gamma)}\setminus \Delta$ and $t>0$ such that $\alpha_u(t)\in L_{\rho(\Gamma)}\times L_{\rho(\Gamma)}\setminus \Delta$. For $s\in \R$, we have $H^u_{t+s}=H_s^{\alpha_u'(t)}\circ H_t^u$, hence $k_u(t+s)=k_u(t)+(H_t^u)^*K_{\alpha_u'(t)}(s)$, which we can write: $$c(u)(t+s)^2<u,v>^2=c(u)t^2<u,v>^2+c(\alpha_u'(t))s^2<dH_t^u(v),\alpha_u'(t)>^2$$ \indent  By computing the derivative with respect to $s$ at $s=0$ on both sides, we obtain $c(u)=0$, i.e. $k_u(t)=0$.\end{proof}

Note that this result could also have been proven with the explicit computation of the Schwarzian derivative in Chapter 4, section 3 of \cite{these}.\\
\indent We can now prove Theorem \ref{rigidity_schwarzian}, that can be slightly reformulated:\\

\noindent \textbf{Theorem \ref{rigidity_schwarzian}.} \emph{If $\rho: \Gamma \to \Diff(\Ss^1)$ is a non elementary representation that preserves a smooth volume form on $\mathcal C$, then there is a projective structure on $\Ss^1$, equivalent to the standard structure on $\R\proj^1$, such that $S(\rho(\gamma))(x)=0$ for all $\gamma \in \Gamma$ and $x\in L_{\rho(\Gamma)}$.}

\begin{proof} Let $I$ be a connected component of  $\Ss^1\setminus L_{\rho(\Gamma)}$ and let $x_0,x_-,x_+\in I$ be such that $x_-<x_0<x_+<x_-$ and the interval consisting of points $x$ such that $x_-<x<x_+<x_-$ is included in $I$. We can choose a parametrization $\p:\Ss^1\setminus \{x_0\}$  of the horizontal geodesic $\Ss^1\setminus \{x_0\}\times  \{x_0\}$ such that the image $\p(\Ss^1\setminus \intoo{x_-}{x_+})$ is equal to $\intff{-1}{1}$. \\
\indent Let $\psi:\Ss^1\to \R\proj^1$ be a diffeomorphism such that the restriction of $\psi$ to $\Ss^1\setminus \intoo{x_-}{x_+}$ is equal to the restriction of $\p$. It equips $\Ss^1$ with a projective structure equivalent to the standard structure on $\R\proj^1$.\\
\indent Let $x\in L_{\rho(\Gamma)}$ and let $\gamma \in \Gamma$. Since $L_{\rho(\Gamma)} \subset \Ss^1\setminus \intoo{x_-}{x_+}$, the projective structure is defined by $\p$. Hence it is sent by $\gamma$ to a parametrization of another horizontal geodesic, and the Schwarzian derivative of $\gamma$ at $x$ is the Schwarzian derivative of the holonomy at $x$, and it is equal to $0$.
\end{proof}

\section{Actions on the circle and flows in dimension $3$} \label{sec:flow}

\indent The rest of this paper is dedicated to Theorem \ref{non_Fuchsian}, which we recall (convex cocompact groups will be defined in subsection \ref{subsection:convex}):\\

\noindent \textbf{Theorem \ref{non_Fuchsian}.} \emph{Let $\rho_0: \Gamma \to \emph{\PSL}(2,\R)$ be a convex cocompact representation and let $h\in \emph{\Homeo}(\Ss^1)$ be such that $h_{/L_{\rho_0(\Gamma)}}=Id$ and  $\rho_1 =h \rho_0 h^{-1} $ has values in $\emph \Diff(\Ss^1)$. Then $\rho_1$ preserves a $C^2$ volume form on $\mathcal C$.}\\

\indent The main ingredient in this proof is to construct a flow on a $3$-manifold (a deformation of the geodesic flow on $\rm{T}^1\h^2/\rho_0(\Gamma)$) that has a transverse structure given by $\rho_1$. This construction follows an idea of Ghys used in two different settings. The first one, found in \cite{Gh93}, was to show a rigidity theorem for actions of surface groups on the circle, and the second was the construction of (the only) exotic Anosov flows with smooth weak stable and weak unstable foliations on 3-manifolds  in \cite{Gh92}, called quasi-Fuchsian flows. However, Ghys used a local construction (given a certain atlas on  $\rm{T}^1\h^2/\rho_0(\Gamma)$), whereas we will take a global approach.\\
\indent We will see in \ref{subsec:constructing_example} that there are some non differentially Fuchsian examples satisfying the hypothesis of Theorem \ref{non_Fuchsian}.

\subsection{Hyperbolic flows} \label{subsec:hyperbolic_flows}  Let us recall a few basic notions of hyperbolic flows. Let $\p^t$ be a complete flow generated by a vector field $X$ on a manifold $M$. We say that a compact invariant set $K\subset M$ is \textbf{hyperbolic} if there are positive constants $C,\lambda$ and a decomposition of tangent spaces $T_xM=E^s_x\oplus E^u_x\oplus \R.X$ for each $x\in K$ such that:
$$ \forall x\in K~\forall v\in E^s_x~\forall t\ge 0~~\Vert D\p^t_x(v)\Vert \le C e^{-\lambda t} \Vert v\Vert $$
$$\forall x\in K~\forall v\in E^u_x~\forall t\le 0~~\Vert D\p^t_x(v)\Vert \le C e^{\lambda t} \Vert v\Vert
$$
\indent The norm $\Vert . \Vert$ denotes the norm given by any Riemannian metric on $M$ (since $K$ is compact, the definition does not depend on the choice of a Riemannian metric). If the whole manifold $M$ is a hyperbolic set, then we say that $\p^t$ is an Anosov flow.\\
\indent Let $\p^t$ be a smooth flow on a manifold $M$. If $K\subset M$ is a compact hyperbolic set and $x\in K$, then we define the \textbf{stable and unstable manifolds} through $x$:
$$ W^s(x)=\{ z\in M \vert d(\p^t(x),\p^t(z))\underset{t\to +\infty}{\longrightarrow} 0\}$$
$$W^u(x)=\{ z\in M \vert d(\p^t(x),\p^t(z))\underset{t\to -\infty}{\longrightarrow}  0\}$$
\indent The Stable Manifold Theorem states that they are submanifolds of $M$ tangent to $E^s$ and $E^u$ at $x$ (see \cite{HP}).\\
\indent The most important fact for us is that the limit $d(\p^t(x),\p^t(z))\to 0$ is a uniformly decreasing exponential: for all compact set $A$ and all $\e>0$, there is a constant $C'>0$ such that:
$$\forall x\in K ~\forall z\in W^s(x)\cap A~\forall t\ge 0~~ d(\p^t(x),\p^t(z))\le C' e^{-(\lambda-\e) t} $$
$$\forall x\in K ~\forall z\in W^u(x)\cap A~\forall t\le 0~~ d(\p^t(x),\p^t(z))\le C' e^{(\lambda-\e) t} $$
\indent We will denote by $W^s(K)$ (resp. $W^u(K)$) the union $W^s(K)=\bigcup_{x\in K} W^s(x)$ (resp. $W^u(K)=\bigcup_{x\in K} W^u(x)$).

\subsection{A cohomological reformulation}  

Searching for an invariant volume form is equivalent to solving a cohomological equation. Let $\omega_0$ be a volume form on $\mathcal C$. Any other volume form on $\mathcal C$ is a multiple of $\omega_0$, hence if $\gamma \in \Gamma$, then we can write $\rho(\gamma)^*\omega_0=e^{-\alpha_{\gamma}}\omega_0$.  The chain rule shows that $\alpha_{\gamma}$ satisfies the cocycle relation $\alpha_{\gamma'\gamma}=\alpha_{\gamma'}\circ \rho(\gamma) + \alpha_{\gamma}$.\\
\indent Let $\omega=e^{\sigma}\omega_0$ be a volume form on $\mathcal C$. We can compute the pull back $\rho(\gamma)^*\omega=e^{\sigma\circ\rho(\gamma)}\rho(\gamma)^*\omega_0=e^{\sigma\circ\rho(\gamma)-\sigma-\alpha_{\gamma}}\omega$, hence $\omega$ is preserved by $\Gamma$ if and only if $\sigma\circ \rho(\gamma)-\sigma = \alpha_{\gamma}$ for all $\gamma\in \Gamma$. In other words, we wish to show that the cocycle $\alpha_{\gamma}$ is a coboundary. \\
\indent The issue with this formulation of the problem is that we do not know much about the cohomology of $\Gamma$. We will now see how we can translate the problem to a cohomology equation for a hyperbolic flow, which is a much more simple situation. In this setting, a cocycle is a smooth  function $\alpha:M\to \R$ (where $M$ is the manifold on which we study a flow $\p^t$), and we look for a smooth function $\sigma :M\to \R$ such that $\sigma(\p^t(x))-\sigma(x)=\int_0^t\alpha(\p^s(x))ds$ for all $(x,t)\in M\times \R$.\\
\indent There is a first necessary condition for the existence of a solution: if $x\in \mathrm{Per}(\p)$, i.e. if there is $T>0$ such that $\p^T(x)=x$, then $\int_0^T\alpha(\p^s(x))ds=0$.  Livšic's Theorem states that this condition is sufficient in order to find a solution on a compact hyperbolic set.

\begin{theo} Let $\p^t$ be a smooth flow on a manifold $M$, and let $K$ be a compact hyperbolic set, such that the action on $K$ has a dense orbit. If $\alpha : K\to \R$ is a H\"older continuous function such that $\int_0^T\alpha(\p^s(x))ds=0$ for all $x\in K$ such that $\p^T(x)=x$, then there is a unique H\"older continuous function $\sigma :K\to \R$ such that $\sigma(\p^t(x))-\sigma(x)=\int_0^t\alpha(\p^s(x))ds$ for all $(x,s)\in K\times \R$. \end{theo}

As stated, the proof can be found in \cite{KH} (Livšic's work in \cite{livsic} deals with Anosov flows on compact manifolds). We will discuss the different versions of  Livšic's Theorem (especially concerning  regularity conditions)  in section \ref{sec:spectrally}.\\
\indent However,  Livšic's Theorem will not be of any use in the proof of Theorem \ref{non_Fuchsian}, because we will already have a solution on the hyperbolic set (but we will use it in section \ref{sec:spectrally} for Theorem  \ref{spectrally}). Instead, we will show that given a solution on a compact hyperbolic set $K$, we can extend it to $W^s(K)\cup W^u(K)$. When translating the problem back to the action on $\mathcal C=\Ss^1\times \Ss^1\setminus \Delta$, this will give a volume form invariant at points of $L_\Gamma\times \Ss^1\cup\Ss^1\times L_\Gamma$, and there will still be some work involved in order to extend the solution to $\mathcal C$ (which is the content of section \ref{sec:non_fuchsian}).

\subsection{Convex cocompact groups and geodesic flows} \label{subsection:convex} Let $\Gamma\subset \PSL(2,\R)$ be a discrete non elementary subgroup such that the limit set $L_\Gamma$ is a Cantor set. The \textbf{convex hull} of $\Gamma$ is the subset $C_\Gamma$ of $\h^2$ bounded by geodesics joining fixed points of hyperbolic elements of $\Gamma$. We say that $\Gamma$ is \textbf{convex cocompact} if $C_\Gamma /\Gamma$ is compact. A particular case of Ahlfors' Finiteness Theorem (see \cite{ahlfors} or \cite{bers}) states that any finitely generated discrete subgroup of $\PSL(2,\R)$ with only hyperbolic elements is convex cocompact.\\
\indent If $\Gamma\subset \PSL(2,\R)$ is convex cocompact, then denote by $\p^t$ the geodesic flow on $\mathrm{T}^1\h^2/\Gamma$ (remark that even if $\h^2/\Gamma$ is not a manifold, the unit bundle $\mathrm{T}^1\h^2/\Gamma$ always is when $\Gamma$ is discrete).\\
\indent The \textbf{non wandering set} $\Omega_\p$ of a flow is the set of points $x$ such that there are sequences $x_n\to x$ and  $t_n\to \infty$ satisfying $\p^{t_n}(x_n)\to x$. For the geodesic flow,  $\Omega_\p$ can be described as follows: its lift to $\rm{T}^1\h^2$ is the set of vectors tangent to a geodesic that lies entirely in $C_\Gamma$. The important property of $\p^t$ is that it is an Axiom A flow: $\Omega_\p$ is a compact hyperbolic set for $\p^t$, and it is equal to the closure of periodic orbits $\mathrm{Per}(\p)$ (Axiom A flows are a generalization of Anosov flows that can be defined even on non compact manifolds). We will now use a presentation of the geodesic flow that is particularly convenient when we define perturbations.

\indent Let $\Sigma_3 = \{(x_-,x_0,x_+)\in (\Ss^1)^3 \vert x_-<x_0<x_+<x_-\}$ be the set of ordered triples of $\Ss^1$. We can identify $\rm{T}^1\h^2$ and $\Sigma_3$ in the following way: given a unit vector $v\in \rm{T}^1\h^2$, we consider $x_-$ and $x_+$ the limits at $-\infty$ and $+\infty$ of the geodesic given by $v$, and $x_0$ is the limit at $+\infty$ of the geodesic passing through the base point of $v$ in an orthogonal direction, oriented to the right of $v$ (see Figure \ref{fig:geodesics}).\\
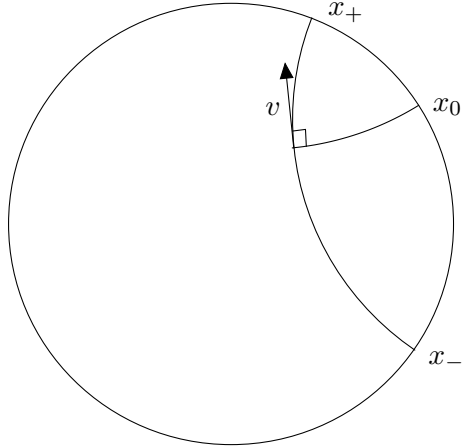
\begin{figure}[h]
\begin{tikzpicture}[line cap=round,line join=round,>=triangle 45,x=1cm,y=1cm,scale=0.55]
\clip(-9,-4.98) rectangle (16.96,6.3);
\draw(2.52,0.82) circle (5.32cm);
\draw [shift={(10.74,3.34)}] plot[domain=2.77:4.11,variable=\t]({1*6.75*cos(\t r)+0*6.75*sin(\t r)},{0*6.75*cos(\t r)+1*6.75*sin(\t r)});
\draw [shift={(3.3,9.63)}] plot[domain=4.81:5.27,variable=\t]({1*7.01*cos(\t r)+0*7.01*sin(\t r)},{0*7.01*cos(\t r)+1*7.01*sin(\t r)});
\draw [->] (4.02,2.65) -- (3.81,4.72);
\draw (3.1,4) node[anchor=north west] {$v$};
\draw (4.6,6.4) node[anchor=north west] {$x_+$};
\draw (7,-2) node[anchor=north west] {$x_-$};
\draw (7.1,4.1) node[anchor=north west] {$x_0$};
\draw (3.99,3.06)-- (4.3,3.1);
\draw (4.3,3.1)-- (4.32,2.69);
\end{tikzpicture}
\caption{Identification between $\rm{T}^1\h^2$ and $\Sigma_3$} \label{fig:geodesics}
\end{figure}

\indent On $\Sigma_3$, the geodesic vector field is a rescaling of the constant vector field $(0,1,0)$, and the action $\alpha$ of $\PSL(2,\R)$ is the diagonal action. The geodesic flow $\p^t$ is defined on the quotient manifold $M=\Sigma_3/\alpha(\Gamma)\approx \rm{T}^1\h^2/\Gamma$. The image of a point $(x_-,x_0,x_+)$ in $M$ is in $\Omega_\p$ if and only if $(x_-,x_+)\in L_{\Gamma}\times L_{\Gamma}$, and it is in $\rm{Per}(\p)$ if and only if  $(x_-,x_+)$ is the pair of fixed points of an element $\gamma\in \Gamma$.

\subsection{The flow associated to $\rho_1$}
From now on, we consider a convex cocompact representation $\rho_0:\Gamma \to \PSL(2,\R)$  and another representation $\rho_1:\Gamma \to \Diff(\Ss^1)$ such that there is $h\in \Homeo(\Ss^1)$ satisfying $h_{/L_{\rho_0(\Gamma)}}=Id$ and $\rho_1=h\rho_0 h^{-1}$. Let us start by remarking that $L_{\rho_0(\Gamma)}$ is a compact invariant set for $\rho_1$. Because of the uniqueness of the minimal invariant compact set, we see that $L_{\rho_1(\Gamma)}\subset L_{\rho_0(\Gamma)}$. Since the actions $\rho_0$ and $\rho_1$ restricted to $L_{\rho_0(\Gamma)}$ are equal and have dense orbits, we have $L_{\rho_1(\Gamma)}= L_{\rho_0(\Gamma)}$. We will  call this set $L_{\Gamma}$.\\
 \indent We are now going to construct a flow $\psi^t$ on a $3$-manifold $N$ that will have the same relation to $\rho_1$ as the geodesic flow $\p^t$ on $M=\rm{T}^1\h^2/\rho_0(\Gamma)$  has with $\rho_0$.  We  consider $\Sigma=\{(x_-,x_0,x_+)\in (\Ss^1)^3\vert x_-<h^{-1}(x_0)<x_+<x_-\}$, and the action  $\alpha_1$ of $\Gamma$ on $\Sigma$ given by:
$$\alpha_1(\gamma)(x_-,x_0,x_+)=(\rho_1(\gamma)(x_-),\rho_0(\gamma)(x_0),\rho_1(\gamma)(x_+))$$  \indent The quotient $N$ is a smooth manifold homeomorphic to $M$: consider the map $\tilde H :\Sigma_3\to \Sigma$ defined by $\tilde H(x_-,x_0,x_+)=(h(x_-),x_0,h(x_+))$. It is a homeomorphism satisfying $\tilde H\circ \alpha_0=\alpha_1\circ\tilde H$ that is  differentiable in restriction to  $L_{\Gamma}\times\Ss^1\times L_{\Gamma}$. It induces a homeomorphism $H:M\to N$.\\
\indent  The projection on $N$ of the constant vector field $(0,1,0)$ on $\Sigma$  can be reparametrized into a smooth flow $\psi^t$. The homeomorphism $H$ sends $\p^t$ to a reparametrization of $\psi^t$ and  is a diffeomorphism from $\Omega_{\p}$ to $\Omega_{\psi}$ (recall that  the non wandering set $\Omega_\Phi$ of the flow $\Phi^t$ is the set of points $x$ such that there are sequences $x_n\to x$ and  $t_n\to \infty$ satisfying $\Phi^{t_n}(x_n)\to x$). From this we deduce that $\Omega_{\psi}$ is a compact hyperbolic set for $\psi^t$. If the image $x\in N$ of $(x_-,x_0,x_+)\in \Sigma$ is in $\Omega_{\psi}$, then the stable (resp. unstable) manifold of $x$ is the set of images of points $(y_-,y_0,y_+)$ such that $y_+=x_+$ (resp. $y_-=x_-$).\\
\indent The classical result for solving cohomological equation for hyperbolic flows is Livšic's Theorem. However, it only provides solutions on the hyperbolic set, and we already have an invariant volume on $\Omega_{\psi}$ (because the flow $\psi^t$ and the geodesic flow $\p^t$ are differentially conjugate on their non wandering sets). The hyperbolicity gives us an extension to $W^s(\Omega_{\psi})\cup W^u(\Omega_{\psi})$, which consists of projections of points $(x_-,x_0,x_+)\in \Sigma$ such that $x_-\in L_{\Gamma}$ or $x_+\in L_{\Gamma}$.

\begin{lemma} \label{invariant_stable_manifold} There is a smooth volume form $\omega_1$ on $N$ that is invariant under $\psi^t$ at points of $W^s(\Omega_{\psi})\cup W^u(\Omega_{\psi})$. \end{lemma}

\begin{proof}  The differentiable conjugacy on the non wandering set  implies that there is a smooth volume form $\omega_0$ on $N$ that is preserved by the flow at points of the non wandering set. Hence, if $\psi^{t*}\omega_0=e^{-A(t,x)}\omega_0$ and $\alpha(x)=\frac{\partial A}{\partial t}(0,x)$, then $\alpha =0$ on $\Omega_{\psi}$. We will now construct a smooth function $\sigma$ on $N$ such that $\sigma(\psi^t(x))-\sigma(x)=\int_0^t\alpha(\psi^s(x))ds$ for all $x\in W^s(\Omega_{\psi})\cup W^u(\Omega_{\psi})$, so that $\omega_1=e^\sigma \omega_0$ meets our requirements.\\ 
\indent If $x\in W^s(z)$ with $z\in \Omega_{\psi}$, and if we have found such a function $\sigma$, then $\sigma(\psi^t(x))\approx \sigma(\psi^t(z))=0$ for $t$ large enough, hence $\sigma(x)=-\int_0^{\infty}\alpha(\psi^t(x))dt$. We will use this formula as a definition of $\sigma$. If it is well defined, then it satisfies the cohomology equation. \\
\indent Let $C>0$ be such that $d(\psi^t(x),\psi^t(z))\le Ce^{-t}$ (locally $C$ can be chosen independently from $x$ and $z$). Let $k$ be a Lipschitz constant for $\alpha$ in a neighbourhood $U$ of $\Omega_{\psi}$. For $t$ such that $\psi^t(x)\in U$ (which is locally uniform in $x$), we have: $$\vert \alpha(\psi^t(x))\vert \le \underbrace{\vert \alpha(\psi^t(z))\vert}_{=0} + k\underbrace{d(\psi^t(x),\psi^t(z))}_{\le Ce^{-t}}$$
\indent This gives us uniform convergence, hence $\sigma$ is well defined and continuous. By applying the same reasoning with negative times, we define $\sigma$ on $W^u(\Omega_{\psi})$.\\
\indent We now wish to see that it is differentiable (i.e. it is the restriction to $W^s(\Omega_{\psi})\cup W^u(\Omega_{\psi})$ of a differentiable function). Since the problem of differentiation is local, we can assume that the underlying manifold is $\R^3$ (so that tangent vectors at $z$ and at $x$ can be identified). Let $k'$ be a Lipschitz constant for $d^2\alpha$ in $U$. For $t$ large enough, we have: \begin{eqnarray*} d\alpha_{\psi^t(x)}(d\psi^t_x(v))&= & \underbrace{d\alpha_{\psi^t(z)}(d\psi^t_x(v))}_{=0}  \\ &+&\int_0^1\underbrace{d^2\alpha_{\psi^t(z)+s(\psi^t(x)-\psi^t(z))}}_{\le k'Ce^{-t}}(\underbrace{\psi^t(x)-\psi^t(z)}_{\le Ce^{-t}},\underbrace{d\psi^t_x(v)}_{\le C'e^t})ds\end{eqnarray*} hence $$\vert d\alpha_{\psi^t(x)}(d\psi^t_x(v))\vert \le C''e^{-t}$$
and $\sigma$ is $C^1$. By iterating this reasoning (to estimate $d^k\alpha$ we have to use a Taylor development at order $2k$, so that we have $k$ terms dominated by $e^t$ and $k+1$ terms dominated by $e^{-t}$), we show that $\sigma$ is $C^{\infty}$.
\end{proof}

\section{Non Fuchsian examples} \label{sec:non_fuchsian}

\subsection{Going back from  $N$ to  $\mathcal C$} Now that we have found an invariant volume form on a larger set for the  flow $\psi^t$, we need to translate it in terms of the action on $\mathcal C$.
\begin{lemma} \label{back_to_C} If there is a $C^r$ volume form $v$ on $N$ preserved by $\psi^t$ at points of $W^s(\Omega_{\psi})\cup W^u(\Omega_{\psi})$, then there is a $C^r$ volume form $\omega_2$ on $\mathcal C$ preserved by $\rho_1(\Gamma)$  at points of $L_{\Gamma}\times \Ss^1\cup \Ss^1\times L_{\Gamma}$. \end{lemma}

\begin{proof}  We have defined a smooth volume form $\omega_1=e^{\sigma}\omega_0$ that is invariant at points of $W^s(\Omega_{\psi})\cup W^u(\Omega_{\psi})$. Let $\tilde \omega_1$ be its lift to $\Sigma_3$ and write: $$\tilde \omega_1=\tilde \omega_1(x_-,x_0,x_+)dx_-\wedge dx_0\wedge dx_+$$ \indent  If $x_-$ or $x_+$ is in $L_{\Gamma}$, then the image in $N$ is in $W^s(\Omega_{\psi})\cup W^u(\Omega_{\psi})$, and the invariance under the  flow $\psi^t$ gives us $\tilde \omega_1(x_-,x_0,x_+) = \tilde \omega_1(x_-,x'_0,x_+)$ for all $x'_0$ such that $(x_-,x'_0,x_+)\in \Sigma_3$. \\
\indent Choose a smooth map $i_0:\mathcal C\to \Ss^1$ such that $(x_-,i_0(x_-,x_+),x_+)\in \Sigma_3$ for all $(x_-,x_+)\in \mathcal C$ (such as a convex combination of $x_-$ and $x_+$), and let $\omega_2(x_-,x_+)=\tilde \omega_1(x_-,i_0(x_-,x_+),x_+)$ for $(x_-,x_+)\in \mathcal C$. If $x_-$ or $x_+$ is in $L_{\Gamma}$ and $\gamma\in \Gamma$, then the invariance under $\psi^t$ gives us: 
\begin{eqnarray*} &~& \omega_2(\rho_1(\gamma)(x_-),\rho_1(\gamma)(x_+)) \rho_1(\gamma)'(x_-)\rho_1(\gamma)'(x_+) \\&=& \tilde \omega_1 ( \rho_1(\gamma)(x_-), i_0(\rho_1(\gamma)(x_-),\rho_1(\gamma)(x_+)),\rho_1(\gamma)(x_+)) \rho_1(\gamma)'(x_-)\rho_1(\gamma)'(x_+)\\  &=& \tilde \omega_1 ( \rho_1(\gamma)(x_-), \rho_1(\gamma)(i_0(x_-,x_+)),\rho_1(\gamma)(x_+)) \rho_1(\gamma)'(x_-)\rho_1(\gamma)'(x_+)\\ &=& \tilde \omega_1(x_-,i_0(x_-,x_+),x_+) \\&=& \omega_2(x_-,x_+) \end{eqnarray*}
\indent We have defined a smooth volume form $\omega_2$ on $\mathcal C$ that is $\rho_1(\Gamma)$-invariant at points of $(L_{\Gamma}\times \Ss^1\cup \Ss^1\times L_{\Gamma})\setminus \Delta$.
\end{proof}

\subsection{Extension to vertical strips}

The first step in extending $\omega_2$ to all of $\mathcal C$ is to extend it to vertical strips delimited by elements of $L_\Gamma$, so that we only need to deal with invariance under one element of the group.

\begin{lemma} \label{vertical} Let $I$ be a connected component of $\Ss^1\setminus L_{\Gamma}$, and let $\gamma\in \Gamma$ be a generator of its stabilizer. There is a smooth volume form $\omega$ on $\overline I\times \Ss^1\setminus \Delta$ that is invariant by $\gamma$ and that is equal to $\omega_2$ on $L_{\Gamma}\times \Ss^1\cup \Ss^1\times L_{\Gamma}$. \end{lemma} 
 

\begin{proof} By Proposition \ref{hyperbolic}, there is a smooth volume form $\omega_{\gamma}$ on $\mathcal C$ that is invariant under $\rho_1(\gamma)$.\\ 
\indent Let $a\in L_{\Gamma}\setminus \overline I$. The interval $\intfo{a}{\rho_1(\gamma)(a)}$ is a fundamental domain for the action of $\gamma$ on $\Ss^1 \setminus \overline I$, i.e. for every $y\in \Ss^1 \setminus \overline I$ there is a unique $n_y\in \Z$ such that $\rho_1(\gamma^{n_y})(y)\in \intfo{a}{\rho_1(\gamma)(a)}$. We set $\omega=\omega_2$ on $\overline I\times \intfo{a}{\rho_1(\gamma)(a)}$ and extend $\omega$ to  $ \overline I \times (\Ss^1\setminus \overline I)$ by using the equivariance formula: $$\frac{\omega(x,y)}{\omega_2(\rho_1(\gamma^{n_y})(x),\rho_1(\gamma^{n_y})(y))}= \rho_1(\gamma^{n_y})'(x) \rho_1(\gamma^{n_y})'(y)$$
\indent We have to show that $\omega$ is smooth. First, remark that it is continuous on $\overline I\times \intfo{a}{\rho_1(\gamma)(a)}$: if $(x_n,y_n)\to (a,y)$ with $\rho_1(\gamma)(x_n)\in \intfo{a}{\rho_1(\gamma)(a)}$, using  $a\in L_{\Gamma}$, we see that the volume $\omega_2$ is preserved at $(a,y)$ and we get: \begin{eqnarray*} \omega(x_n,y_n)&=&\omega_2(\rho_1(\gamma)(x_n),\rho_1(\gamma)(y_n))\rho_1(\gamma)'(x_n)\rho_1(\gamma)'(y_n)\\ &\to & \omega_2(\rho_1(\gamma)(a),\rho_1(\gamma)(y))\rho_1(\gamma)'(a)\rho_1(\gamma)'(y)\\ &  ~& ~~~~~~=\omega_2(a,y)=\omega(a,y) \end{eqnarray*}
\indent This shows that $\omega$ is continuous on $\overline I \times (\Ss^1\setminus \overline I)$. For the derivatives, we have:
\begin{eqnarray*} \frac{\partial \omega}{\partial x}(x_n,y_n)&=& \frac{\partial \omega_2}{\partial x}(\rho_1(\gamma)(x_n),\rho_1(\gamma)(y_n)) \rho_1(\gamma)'(x_n)^2\rho_1(\gamma)'(y_n) \\& & ~~~+\omega_2(\rho_1(\gamma)(x_n),\rho_1(\gamma)(y_n))\rho_1(\gamma)''(x_n)\rho_1(\gamma)'(y_n)\\
&\to & \frac{\partial \omega_2}{\partial x}(\rho_1(\gamma)(a),\rho_1(\gamma)(y)) \rho_1(\gamma)'(a)^2\rho_1(\gamma)'(y)\\ & & ~~~ + \omega_2(\rho_1(\gamma)(a),\rho_1(\gamma)(y))\rho_1(\gamma)''(a)\rho_1(\gamma)'(y)\\
& & ~~~=\frac{\partial \omega_2}{\partial x}(a,y) = \frac{\partial \omega}{\partial x}(a,y)
\end{eqnarray*}
\indent The last line comes from the fact that the derivatives of $\omega_2$ satisfy the associated equivariance relations on $L_{\Gamma}\times \Ss^1\cup \Ss^1\times L_{\Gamma}$. This is true because all points of $L_{\Gamma}$ are accumulation points (it is a Cantor set). The same can be applied to all the derivatives, which shows that $\omega$ is smooth on $\overline I \times (\Ss^1\setminus \overline I)$. \\
\indent If $(x_k,y_k)\to (x,y)\in \mathcal C$ with $y\in \partial I$, then set $n_k=n_{y_k}$, as well as $u_k=\rho_1(\gamma^{n_k})(x_k)$ and $v_k=\rho_1(\gamma^{n_k})(y_k)$. By definition, we have: $$\omega(x_k,y_k)=\omega_2(u_k,v_k)\rho_1(\gamma^{n_k})'(x_k)\rho_1(\gamma^{n_k})'(y_k)$$
\indent Since $\omega_{\gamma}$ is invariant under $\rho_1(\gamma)$, we have: $$\rho_1(\gamma^{n_k})'(x_k)\rho_1(\gamma^{n_k})'(y_k) = \frac{\omega_{\gamma}(x_k,y_k)}{\omega_{\gamma}(u_k,v_k)}$$
\indent These two equalities give us: $$\omega(x_k,y_k) =\frac{\omega_2(u_k,v_k)}{\omega_{\gamma}(u_k,v_k)}\omega_{\gamma}(x_k,y_k)$$
\indent The continuity of $\omega_{\gamma}$ gives us $\omega_{\gamma}(x_k,y_k)\to \omega_{\gamma}(x,y)$.\\
\indent Since $y_k\to y\in \partial I$, we have $n_k\to \infty$ and $u_k\to u$ where $u$ is the other extremal point of $I$. By using the uniform continuity of $\omega_2$ and $\omega_{\gamma}$ on $\overline I \times \intff{a}{\rho_1(\gamma)(a)}$, we obtain:
$$\omega(x_k,y_k) \sim \frac{\omega_2(u,v_k)}{\omega_{\gamma}(u,v_k)} \omega_{\gamma}(x,y)$$

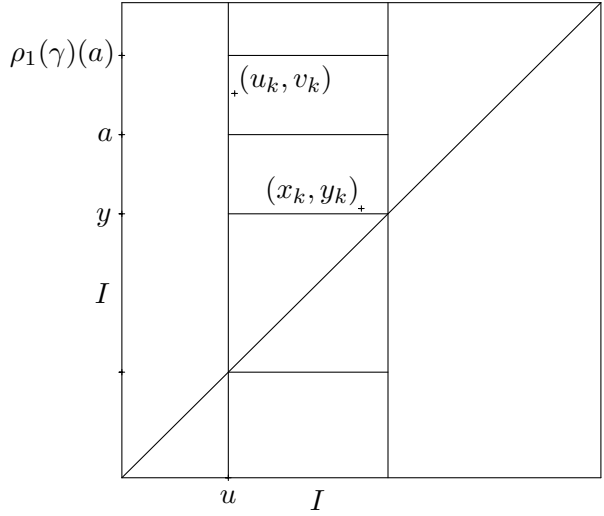
\begin{figure}[h]
\begin{tikzpicture}[line cap=round,line join=round,>=triangle 45,x=1.0cm,y=1.0cm,scale=0.7]
\clip(-4.9,-4.7) rectangle (16.9,5.4);
\draw (-1,5)-- (8,5);
\draw (8,5)-- (-1,-4);
\draw (-1,-4)-- (-1,5);
\draw (-1,-4)-- (8,-4);
\draw (8,-4)-- (8,5);
\draw (1,5)-- (1,-4);
\draw (4,-4)-- (4,5);
\draw (1,1)-- (4,1);
\draw (4,-2)-- (1,-2);
\draw (1,4)-- (4,4);
\draw (1,2.5)-- (4,2.5);
\draw (-1.7,-0.12) node[anchor=north west] {$I$};
\draw (-1.65,2.8) node[anchor=north west] {$a$};
\draw (-1,1.3) node[anchor=north east] {$y$};
\draw (-3.3,4.5) node[anchor=north west] {$\rho_1(\gamma)(a)$};
\draw (1,4) node[anchor=north west] {$(u_k,v_k)$};
\draw (1.5,1.9) node[anchor=north west] {$(x_k,y_k)$};
\draw (2.34,-4.04) node[anchor=north west] {$I$};
\draw (0.65,-4.04) node[anchor=north west] {$u$};
\begin{scriptsize}
\draw [color=black] (-1,2.5)-- ++(-1.5pt,0 pt) -- ++(3.0pt,0 pt) ++(-1.5pt,-1.5pt) -- ++(0 pt,3.0pt);
\draw [color=black] (-1,4)-- ++(-1.5pt,0 pt) -- ++(3.0pt,0 pt) ++(-1.5pt,-1.5pt) -- ++(0 pt,3.0pt);
\draw [color=black] (-1,1)-- ++(-1.5pt,0 pt) -- ++(3.0pt,0 pt) ++(-1.5pt,-1.5pt) -- ++(0 pt,3.0pt);
\draw [color=black] (-1,-2)-- ++(-1.5pt,0 pt) -- ++(3.0pt,0 pt) ++(-1.5pt,-1.5pt) -- ++(0 pt,3.0pt);
\draw [color=black] (3.5,1.1)-- ++(-1.5pt,0 pt) -- ++(3.0pt,0 pt) ++(-1.5pt,-1.5pt) -- ++(0 pt,3.0pt);
\draw [color=black] (1,-4)-- ++(-1.5pt,0 pt) -- ++(3.0pt,0 pt) ++(-1.5pt,-1.5pt) -- ++(0 pt,3.0pt);
\draw [color=black] (1.12,3.28)-- ++(-1.5pt,0 pt) -- ++(3.0pt,0 pt) ++(-1.5pt,-1.5pt) -- ++(0 pt,3.0pt);
\end{scriptsize}
\end{tikzpicture}
\caption{Defining $\omega$ on vertical strips} \label{fig:vertical}
\end{figure}

\indent We now only have to deal with the restrictions of $\omega_2$ and $\omega_{\gamma}$  to the axes $\{u\} \times \Ss^1\cup \Ss^1\times \{y\}$ (see Figure \ref{fig:vertical}), where continuous volume forms invariant under $\rho_1(\gamma)$ are unique up to multiplication by a constant: there is $\lambda>0$ such that $\omega_2(s,t)=\lambda \omega_{\gamma}(s,t)$ whenever $s=u$ or $t=y$. We can finally conclude: $$\omega(x_k,y_k) \to \lambda \omega_{\gamma}(x,y) = \omega_2(x,y)=\omega(x,y)$$
\indent We have shown that $\omega$ is continuous on $(\overline I \times \Ss^1\setminus I )\setminus \Delta$. For the derivatives., we will use the notation $f_x=\frac{\partial \Log \omega}{\partial x}$ and define $f_y$, $f_{xy}$ and so on in the same way. We also define $f_x^\gamma$, $f_y^\gamma$, $f_{xy}^\gamma$, etc\dots ~the derivatives of $\Log \omega_\gamma$. The equivariance relation for $f_x$ is:
$$f_x( x,y)=f_x(\rho_1(\gamma)(x),\rho_1(\gamma)(y))\rho_1(\gamma)'(x) +\frac{\rho_1(\gamma)''(x)}{\rho_1(\gamma)'(x)}$$
\indent We keep the same notations $u_k$, $v_k$ as above, and find:
$$f_x(x_k,y_k)-f_x^\gamma(x_k,y_k) =\rho_1(\gamma^{n_k})'(x_k)(f_x(u_k,v_k)-f_x^\gamma(u_k,v_k))$$
\indent The Mean Value Theorem gives us $u'_k,u''_k \in \intff{u}{u_k}$ such that:
$$f_x(u_k,v_k) - f_x(u,v_k)=(u_k-u) f_{xx}(u'_k,v_k)$$ \indent And:
$$f_x^\gamma(u_k,v_k) - f_x^\gamma(u,v_k)=(u_k-u) f_{xx}^\gamma(u''_k,v_k)$$
\indent The forms $\omega$ and $\omega_\gamma$ are proportional on the axis $\{u\}\times \Ss^1\setminus \{u\}$. This implies that $f_x(u,v_k)=f_x^\gamma(u,v_k)$ (the multiplicative constant disappears because we consider the derivative of the logarithm). Finally, we obtain:
$$f_x(x_k,y_k)-f_x^\gamma(x_k,y_k) =\underbrace{ \rho_1(\gamma^{n_k})'(x_k) (u-u_k)}_{\textrm{bounded}} (\underbrace{f_{xx}(u'_k,v_k)-f_{xx}^\gamma(u''_k,v_k)}_{\to 0})$$
\indent Since $f_x^\gamma$ is continuous, we see that $f_x$ also is. The same technique (applying several times the Mean Value Theorem to get rid of the term $\rho_1(\gamma^{n_k})'(x_k)$ or $\rho_1(\gamma^{n_k})'(y_k)$ which explodes) shows that $\omega$ is smooth on $(\overline I \times \Ss^1\setminus I )\setminus \Delta$.\\
\indent Finally, we can extend $\omega$ to $\overline I \times \Ss^1 \setminus \Delta$ in a similar manner: we fix $\omega$ on a fundamental domain $\intfo{b}{\rho_1(\gamma)(b)}\times \overline I \setminus \Delta$ for some $b\in I$, making sure that the derivatives on the boundary allow the extension on $\overline I \times \overline I\setminus \Delta$ to be smooth.
\end{proof}

\subsection{From vertical strips to $\mathcal C$}

We can now extend $\omega$ to $\mathcal C$. Getting an invariant volume form is not complicated, however its regularity requires some work.
\subsubsection{Continuity} Our proof of the regularity of $\omega$ on vertical strips relied on the existence of a smooth invariant form by any element of $\Gamma$. To deal with the invariance under the whole group, we will need a different method.
\begin{prop} \label{continuous} There is a continuous invariant form $\omega$ on $\mathcal C$ that is invariant under $\rho_1(\Gamma)$ and that is equal to $\omega_2$ on $L_{\Gamma}\times\Ss^1\cup\Ss^1\times L_{\Gamma}$. \end{prop}

\begin{proof} The action of $\Gamma$ on the set of connected components of $\Ss^1\setminus L_{\Gamma}$ has a finite number of orbits (each orbit correspond to a half cylinder in the surface $\h^2/\rho_0(\Gamma)$). Let $I_1,\dots ,I_n$ be a choice of an interval of each orbit. Note that the stabilizer of $I_i$ is always non empty (a generator of the stabilizer corresponds to a closed geodesic bounding a half cylinder in the surface $\h^2/\rho_0(\Gamma)$). By Lemma \ref{vertical}, there is a smooth volume form $\omega$ on $\overline I_i\times \Ss^1 \setminus \Delta$ that is equal to $\omega_2$ in restriction to $L_{\Gamma}\times \Ss^1\cup \Ss^1\times L_{\Gamma}$ and that is invariant under the stabilizer of $I_i$.  If $\gamma \in \Gamma$, then we define $\omega$ on $\rho_1(\gamma) (\overline I_i) \times \Ss^1\setminus \Delta$ to be $\rho_1(\gamma)_*\omega$. This defines a volume form $\omega$ on $\mathcal C$ that is $\rho_1(\Gamma)$-invariant, smooth on all vertical strips $I\times \Ss^1\setminus \Delta$ where $I$ is a connected component of $\Ss^1\setminus L_{\Gamma}$ and equal to $\omega_2$ on $L_\Gamma \times \Ss^1\cup \Ss^1\cup L_\Gamma$.\\
\indent To show that $\omega$ is continuous, assume that $(x_k,y_k)\to (x,y)$ with $x\in L_{\Gamma}$ (if $x\notin L_{\Gamma}$, then there is a connected component $I$ of $\Ss^1\setminus L_{\Gamma}$ such that $x_k\in I$ for $k$ large enough, which gives us $\omega(x_k,y_k)\to \omega(x,y)$, and the same for the derivatives of $\omega$). If $x_k\in L_{\Gamma}$ for all $k$, then $\omega(x_k,y_k)=\omega_2(x_k,y_k)$ and we already have the continuity, hence we can assume that $x_k\notin L_{\Gamma}$ for all $k$. Up to considering a finite number of subsequences, we can assume that there is $\gamma_k\in \Gamma$ such that $u_k=\rho_1(\gamma_k)(x_k) \in I_1$. By composing $\gamma_k$ with an element of the stabilizer of $I_1$, we can take $u_k$  in a compact interval $K\subset I$ (take a fundamental domain $K=\intff{a}{\rho_1(\delta)(a)}$ where $\delta$ is a generator of $\Stab(I_1)$).\\
Let $v_k=\rho_1(\gamma_k)(y_k)$. The definition of $\omega$ is:
$$\omega(x_k,y_k)=\omega(u_k,v_k)\rho_1(\gamma_k)'(x_k)\rho_1(\gamma_k)'(y_k)$$
\indent We have already seen that $\omega$ is continuous on $\overline I_1\times \Ss^1\setminus \Delta$ and $u_k\in I_1$. The problem in finding the limit of $\omega(x_k,y_k)$ is the control of the Jacobian product $\rho_1(\gamma_k)'(x_k)\rho_1(\gamma_k)'(y_k)$. However, we know that $\omega$ is continuous on $L_{\Gamma}\times \Ss^1\cup \Ss^1\times L_{\Gamma}$. We will use this fact to get rid of the derivatives: if $x'_k$ and $y'_k$ are sequences in $L_{\Gamma}$ such that $x'_k\ne y'_k$, $x'_k\ne y_k$ and $x_k\ne y'_k$, then we set $u'_k=\rho_1(\gamma_k)(x'_k)$ and $v'_k=\rho_1(\gamma_k)(y'_k)$. The equivariance equation for $\omega$ gives us:  \begin{equation} \label{equivariance} \frac{\omega(x_k,y_k)}{\omega(x_k,y'_k)}\frac{\omega(x'_k,y'_k)}{\omega(x'_k,y_k)} = \frac{\omega(u_k,v_k)}{\omega(u_k,v'_k)}\frac{\omega(u'_k,v'_k)}{\omega(u'_k,v_k)} \end{equation}  
\indent We are now looking for suitable points $x'_k$ and $y'_k$. Let $I_1=\intoo{a}{b}$, and assume that $v_k$ does not admit $a$ as a limit point (up to considering two subsequences and replacing $a$ by $b$ in the following discussion, we can always assume that it is the case), i.e. that $v_k$ lies in a compact interval $J\subset \Ss^1\setminus\{a\}$. Let $u'_k=a$ and $x'_k=\rho_1(\gamma_k^{-1})(a) \to x$. If $y_k\in L_{\Gamma}$, then we choose $y'_k=y_k$. If $y_k\notin L_{\Gamma}$, then we set $y'_k$ to be an extremal point of the connected component of $\Ss^1\setminus L_{\Gamma}$ containing $y_k$, in a way such that $v'_k=\rho_1(\gamma_k)(y'_k)\in J$.\\
\indent We now have $x'_k\to x$ and $x_k\to x$, which gives:
$$\frac{\omega(x_k,y_k)}{\omega(x_k,y'_k)}\frac{\omega(x'_k,y'_k)}{\omega(x'_k,y_k)}  \sim \frac{\omega(x_k,y_k)}{\omega(x,y'_k)}\frac{\omega(x,y'_k)}{\omega(x,y)} = \frac{\omega(x_k,y_k)}{\omega(x,y)}$$
\indent We wish to show that this quantity converges to $1$ as $k\to \infty$. The compact set  $E=\{b\}\times J \cup K\times \Ss^1\setminus I_1$ of $\mathcal C$  contains the sequences $(u_k,v_k)$,  $(u_k,v'_k)$, $(u'_k,v_k)$ and $(u'_k,v'_k)$. Consequently, the ratio \eqref{equivariance} lies in a compact set of $\intoo{0}{+\infty}$, and it is enough to see that its only possible limit is $1$. If there is a subsequence such that the ratio \eqref{equivariance} converges to $\lambda \in \intoo{0}{+\infty}$, then up to another subsequence, we can assume that the sequence $\gamma_k$ has the convergence property: there are $N,S\in \Ss^1$ such that $\rho_1(\gamma_k)(z)\to N$ for all $z\ne S$. Since $\rho_1(\gamma_k^{-1})(z)\to x$ for all $z\in I_1$, we see that $S$ in necessarily equal to $x$, hence the sequences $v_k$ and $v'_k$ converge to $N\in \Ss^1$. We get: $$ \frac{\omega(u_k,v_k)}{\omega(u_k,v'_k)}\frac{\omega(u'_k,v'_k)}{\omega(u'_k,v_k)} \to \frac{\omega(u,N)}{\omega(u,N)}\frac{\omega(a,N)}{\omega(a,N)} = 1   $$ 
\indent This shows that $\lambda =1$, therefore $\omega(x_k,y_k)\to \omega(x,y)$ and $\omega$ is continuous.
\end{proof}

\subsubsection{Differentiability}
For higher regularity of $\omega$, we will keep the same notations as in the proof of Proposition \ref{continuous} to show that we also have $\frac{\partial^{n+m}\omega}{\partial x^n\partial y^m} (x_k,y_k)\to \frac{\partial^{n+m}\omega_2}{\partial x^n\partial y^m}(x,y)$. By considering the restrictions of $\omega$ to horizontal and vertical circles, this will show that the partial derivatives of $\omega$ are well defined, and that they are continuous, which implies the smoothness of $\omega$. To simplify the calculations, we will use the notation $f_x=\frac{\partial \Log \omega}{\partial x}$ and define $f_y$, $f_{xy}$ and so on in the same way. We will make use repeatedly of an intermediate result.
\begin{lemma} \label{convergence_lemma} Let $g,h:\mathcal C\to \R$ be  functions such that:
\begin{itemize} \item The restrictions of $g$  to vertical strips $\overline I \times \Ss^1\setminus \Delta \to \R$ where $I$ is a connected component of $\Ss^1\setminus L_{\Gamma}$ are $C^1$. \item The restriction of $g$, $h$ and the derivatives of $g$ to $L_{\Gamma}\times \Ss^1\cup \Ss^1\times L_{\Gamma}$ are continuous. \end{itemize} If $h$ is a function such that $h(x_k,y_k)=g(u_k,v_k)\rho_1(\gamma_k)'(x_k) + h_k(x_k)$ for some function $h_k:\Ss^1\to \R$ and for any choice of the sequences $u_k$, $v_k$ defined above, then $h$ is continuous. \end{lemma}
\begin{proof}  The Mean Value Theorem gives us $w_k\in \intff{v_k}{v'_k}$ such that:
$$h(x_k,y_k)-h(x_k,y'_k)=\rho_1(\gamma_k)'(x_k)(v_k-v'_k)\frac{\partial g}{\partial y}(u_k,w_k)$$
\indent A change of variables $s=\rho_1(\gamma_k)(t)$ allows us to compute $v_k-v'_k$, by setting $y_k^t=(1-t)y'_k+ty_k$:
$$v_k-v'_k=\int_{v'_k}^{v_k} ds = \int_{y'_k}^{y_k} \rho_1(\gamma_k)'(t)dt = (y_k-y'_k) \int_0^1\rho_1(\gamma_k)'(y_k^t)dt$$
\indent Let $v_k^t=\rho_1(\gamma_k)(y_k^t)$.
\begin{eqnarray*} h(x_k,y_k)-h(x_k,y'_k)&=&\rho_1(\gamma_k)'(x_k)(y_k-y'_k)\left( \int_0^1\rho_1(\gamma_k)'(y_k^t)dt\right) \frac{\partial g}{\partial y}(u_k,w_k)\\
&=&(y_k-y'_k)\frac{\partial g}{\partial y}(u_k,w_k) \int_0^1\frac{\omega(x_k,y_k^t)}{\omega(u_k,v_k^t)} dt \\ \end{eqnarray*}
\indent This shows that the sequence $h(x_k,y_k)$ is bounded, so all that we have to show is that it only has one limit point. Up to a subsequence, we can assume that $y'_k\to y'\in L_\Gamma$ and that $u_k\to u$.

$$ h(x_k,y_k)-h(x_k,y'_k) \to (y-y')\frac{\partial g}{\partial y}(u,N) \int_0^1\frac{\omega(x,y^t)}{\omega(u,N)} dt $$
\indent We now only have to show that the limit does not depend on $y'$ and $u$. To see this,  we first notice that since the expression is independent on the choice of $u_k$ and $v_k$ (which are defined up to composition with an element of $\Stab(I_1)$), and since $(y-y')\int_0^1\omega(x,y^t)dt\ne 0$, the function $\frac{1}{\omega} \frac{\partial g}{\partial y}$ is invariant under $\rho_1(\Gamma)$. Since it is continuous on $L_{\Gamma}\times \Ss^1\cup \Ss^1\times L_{\Gamma}$, it is constant on this set, and $N\in L_{\Gamma}$. This shows that the limit only depends on $x$, $y$ and $y'$, hence is the same for constant sequences, and it is $h(x,y)-h(x,y')$. Since $h(x_k,y'_k)\to h(x,y')$ (because $y'_k\in L_{\Gamma}$), $h$ is continuous.
\end{proof}

We achieve the proof of Theorem \ref{non_Fuchsian} by showing that $\omega$ is differentiable.

\begin{prop} \label{smooth} $\omega$ is $C^2$. \end{prop}
\begin{proof} If $\gamma\in \Gamma$ and $(x,y)\in \mathcal C$, then the derivative of the equivariance relation $\omega(\rho_1(\gamma)(x),\rho_1(\gamma)(y))\rho_1(\gamma)'(x)\rho_1(\gamma)'(y)=\omega(x,y)$ with respect to $x$ is:
\begin{eqnarray*} \frac{\partial \omega}{\partial x}(x,y) &= &\frac{\partial \omega}{\partial x}(\rho_1(\gamma)(x),\rho_1(\gamma)(y))\rho_1(\gamma)'(x)^2\rho_1(\gamma)'(y)\\ & &~~~~+\omega(\rho_1(\gamma)(x),\rho_1(\gamma)(y))\rho_1(\gamma)''(x)\rho_1(\gamma)'(y)\end{eqnarray*}
\indent Applied to the sequence $(x_k,y_k)$, we get:
\begin{equation} f_x(x_k,y_k)=f_x(u_k,v_k)\rho_1(\gamma_k)'(x_k) + \frac{\rho_1(\gamma_k)''(x_k)}{\rho_1(\gamma_k)'(x_k)} \label{equivariance_fx} \end{equation}

\indent Lemma \ref{convergence_lemma} show that $f_x(x_k,y_k)$ converges to $f_x(x,y)$. For $f_y$, we have:
\begin{equation} f_y(x_k,y_k)=f_y(u_k,v_k)\rho_1(\gamma_k)'(y_k) + \frac{\rho_1(\gamma_k)''(y_k)}{\rho_1(\gamma_k)'(y_k)} \label{equivariance_fy}\end{equation}
\indent Just as in Lemma \ref{convergence_lemma}, we see that $f_y(x_k,y_k)-f_y(x'_k,y_k)\to 0$ (because $x_k-x'_k\to 0$), and we now know that $\omega$ is $C^1$. Derivating once more with respect to $y$, we get:
\begin{eqnarray*} f_{yy}(x_k,y_k)-f_{yy}(x'_k,y_k) &=& \rho_1(\gamma_k)'(y_k)^2(f_{yy}(u_k,v_k)-f_{yy}(u_k,v'_k))\\ & & ~~~~~~~+3(f_y(x_k,y_k)-f_y(x'_k,y_k))\frac{\rho_1(\gamma_k)''(y_k)}{\rho_1(\gamma_k)'(y_k)} \end{eqnarray*}
\indent Since $\rho_1(\gamma_k)'(y_k)\to 0$ (if were not the case, then $\rho_1(\gamma_k)$ would be equicontinuous, which is impossible because $\rho_1(\Gamma)$ is discrete in $\Homeo(\Ss^1)$), we see that the first term tends to $0$. The equivariance formula \eqref{equivariance_fy} for $f_y$ shows that the ratio $\frac{\rho_1(\gamma_k)''(y_k)}{\rho_1(\gamma_k)'(y_k)}$ has a limit as $k\to \infty$, hence is bounded. This shows that $f_{yy}(x_k,y_k)-f_{yy}(x'_k,y_k) \to 0$, i.e. that $f_{yy}$ is continuous.\\
\indent For the crossed derivative $f_{xy}$, we use the derivative with respect to $y$ of \eqref{equivariance_fx}:
\begin{eqnarray*}  f_{xy}(x_k,y_k) &=&f_{xy}(u_k,v_k) \rho_1(\gamma_k)'(x_k)\rho_1(\gamma_k)'(y_k)\\ &=& f_{xy}(u_k,v_k) \frac{\omega(x_k,y_k)}{\omega(u_k,v_k)} \end{eqnarray*} 
\indent Since $\omega$ is continuous, we have:
$$f_{xy}(x_k,y_k)\to f_{xy}(u,N)\frac{\omega(x,y)}{\omega(u,N)}$$
\indent This limit gives the impression that it depends on $u$, however the curvature function $\frac{1}{\omega} f_{xy}$ is $\rho_1(\Gamma)$-invariant, and continuous on $L_{\Gamma}\times \Ss^1\cup \Ss^1\times L_{\Gamma}$, hence constant on this set (the proof of Lemma \ref{curvature_limit_set} can be applied) and the limit does not depend on $u$ (because $N\in L_{\Gamma}$). This shows that $f_{xy}$ is continuous. To get the convergence for $f_{xx}$, we first notice that it is sufficient to show that $f_{xxy}$ converges:
\begin{eqnarray*} f_{xx}(x_k,y_k) &= &f_{xx}(x_k,y'_k) +\int_{y'_k}^{y_k} f_{xxy}(x_k,t) dt \\ &\to & f_{xx}(x,y') +\int_{y'}^{y} f_{xxy}(x,t) dt = f_{xx}(x,y) \end{eqnarray*} 
\indent The reason why we consider $f_{xxy}$ rather than $f_{xx}$ is to get a control on the term $\rho_1(\gamma_k)'(x_k)^2$ by multiplying it with $\rho_1(\gamma_k)'(y_k)$. The equivariance formula is:
\begin{eqnarray*} f_{xxy}(x_k,y_k) 
&=& f_{xxy}(u_k,v_k) \rho_1(\gamma_k)'(x_k)^2\rho_1(\gamma_k)'(y_k) \\ & & ~~~~~~~+f_{xy}(u_k,v_k) \rho_1(\gamma_k)''(x_k) \rho_1(\gamma_k)'(y_k) 
\end{eqnarray*}
\indent If we consider $g=\frac{1}{\omega}f_{xxy}$ and $h=\frac{1}{\omega}f_{xy}$, we can simplify:
$$  g(x_k,y_k)=g(u_k,v_k) \rho_1(\gamma_k)'(x_k) + h(u_k,v_k) \frac{\rho_1(\gamma_k)''(x_k)}{\rho_1(\gamma_k)'(x_k)} $$
\indent The equivariance relation \eqref{equivariance_fx} for $f_x$ allows us to get rid of the term $\frac{\rho_1(\gamma_k)''(x_k)}{\rho_1(\gamma_k)'(x_k)}$:
$$g(x_k,y_k)=\rho_1(\gamma_k)'(x_k) \left(g(u_k,v_k) -f_x(u_k,v_k)h(u_k,v_k)\right)    +f_x(x_k,y_k) h(u_k,v_k)$$
\indent We now set $k=g-f_xh$ so that we have (by using the fact that $h$ is $\rho_1(\Gamma)$-invariant): $$g(x_k,y_k)=k(u_k,v_k)\rho_1(\gamma_k)'(x_k) + f_x(x_k,y_k)h(x_k,y_k)$$  Lemma \ref{convergence_lemma} gives the convergence of the first term, and we have already shown that $f_x$ and $h=\frac{1}{\omega}f_{xy}$ are continuous. This shows that $\omega$ is $C^2$.
\end{proof}
\indent To get a smooth $\omega$, first show that we can get $\frac{\partial^{n+m}}{\partial x^n\partial y^m} \Log \omega$ when $m>n$, then  integrate with respect to $y$ to get all derivatives.

\subsection{Constructing an example} \label{subsec:constructing_example} In order to make  Theorem \ref{non_Fuchsian} relevant, we will see that such examples of groups exist. Start with a Schottky representation $\rho_0: \mathbb F_2=\langle a,b\rangle  \to \PSL(2,\R)$ generated by two hyperbolic elements $\rho_0(a)=\gamma_1,\rho_0(b)=\gamma_2$. Consider two circle diffeomorphisms $\p_1,\p_2$ that are the identity on the limit set $L_{\rho_0(\mathbb F_2)}$, and set $\tilde \gamma_i=\p_i^{-1}\gamma_i\p_i$. We define the  representation $\rho_1: \mathbb F_2 \to \Diff(\Ss^1)$  by $\rho_1(a)=\tilde \gamma_1$ and $\rho_1(b)=\tilde \gamma_2$.

\begin{lemma} $\rho_1$ is differentially Fuchsian if and only if $\p_1=\p_2$. \end{lemma}

\begin{proof} If $\p_1=\p_2$, then $\p_1$ is a differentiable conjugacy between $\rho_0$ and $\rho_1$, so $\rho_1$ is differentially Fuchsian.\\
\indent Assume that $\rho_1$ is differentially Fuchsian. Let $\p\in \Diff(\Ss^1)$ be such that $\p^{-1} \rho_1(\mathbb F_2)\p \subset \PSL(2,\R)$. Up to composing $\p$ with an element of $\PSL(2,\R)$, we can assume that $\p^{-1} \rho_1(a)\p = \rho_0(a)$. This implies that $\p_1^{-1}\circ \p$ commutes with $\gamma_1$, hence that there is $t\in \R$ such that $\p_1^{-1}\circ \p = \gamma_1^t$ (where $\gamma_1^t$ denotes the one parameter subgroup of $\PSL(2,\R)$ generated by $\gamma_1$, see \ref{subsec:hyperbolic} for a proof).  Similarly, there is $s\in \R$ such that $\p_2^{-1}\circ \p =\gamma_2^s$ (an element of the one parameter group generated by $\gamma_2$).\\
\indent The equality $\p_2\circ \gamma_2^s = \p_1\circ \gamma_1^t$ applied to the fixed points of  $\gamma_1$ and $\gamma_2$ shows that $s=t=0$, hence $\p_1=\p_2$.
\end{proof}

\begin{prop} There is $h\in \Homeo(\Ss^1)$ such that $h_{/L_{\rho_0(\mathbb F_2)}}=Id$ and $\rho_1=h\rho_0 h^{-1}$. \end{prop}
\begin{proof} Let $\Ss^1\setminus L_{\rho_0(\mathbb F_2)}=\bigcup_{i\in \N}I_i$ its decomposition into connected components, and let $A\subset \N$ be a fundamental domain for the action of $\mathbb F_2$ on the set of connected components of $\Ss^1\setminus L_{\rho_0(\mathbb F_2)}$. Given $i\in A$,  set $h_{/I_i}$ any homeomorphism that fixes the endpoints of $I_i$ such that $h_{/I_i}\circ \rho_0(\delta)=\rho_1(\delta)\circ h_{/I_i}$ for $\delta$ in the stabilizer of $I_i$. For $\gamma \in \mathbb F_2$, set $h=\rho_1(\gamma) \circ h_{/I_i}\circ \rho_0(\gamma^{-1})$ on $\rho_0(\gamma) (I_i) = \rho_1(\gamma) (I_i)$. This defines an element $h\in \Homeo(\Ss^1)$ that fixes all points of $L_{\rho_0(\mathbb F_2)}$ such that $h^{-1}\rho_1 h = \rho_0$.  \end{proof}
Note that we proved here that $\rho_1(\mathbb F_2)$ remains a free group, which is a general fact for a $C^1$ perturbation of a Schottky group (see \cite{Sullivan}).

\section{Spectrally Möbius-like deformations} \label{sec:spectrally} In the proof of Theorem \ref{non_Fuchsian},  we used the fact that the conjugacy  is the identity on the limit set for two purposes: in order to find an invariant volume form on $L_\Gamma\times\Ss^1\cup\Ss^1\times L_\Gamma \setminus\Delta$, and in order to show that $\Omega_\psi$ is a hyperbolic set. In the case of spectrally Möbius-like actions, we only have an invariant volume form on pairs of fixed points of elements of $\Gamma$.\\
\indent In the context of the flow $\psi$, this means that we need to find an invariant volume form on $\Omega_\psi$, starting with a data on periodic orbits. This is exactly the context of  Livšic's Theorem. However, we still need hyperbolicity for the flow $\psi$, which is why we only prove Theorem \ref{spectrally} for small perturbations of Fuchsian groups.\\
\indent Given a representation $\rho_0:\Gamma \to \Diff(\Ss^1)$ of a finitely generated group $\Gamma$, we say that $\rho : \Gamma \to \Diff(\Ss^1)$ is $C^1$-close to $\rho_0$ if the images under $\rho$ of a system of generators of $\Gamma$ are close to   the images under $\rho_0$ in the $C^1$ topology.\\

\noindent \textbf{Theorem \ref{spectrally}.} \emph{Let $\rho_0: \mathbb F_n \to \PSL(2,\R)$ be  a convex cocompact representation. If $\rho_1: \mathbb F_n \to \Diff(\Ss^1)$ is sufficiently $C^1$-close to $\rho_0$, and if $\rho_1$ is spectrally Möbius-like, then $\rho_1$ is area-preserving. }

\begin{proof} The central argument is the fact that the flow $\psi$ associated to $\rho_1$ is $C^1$-close to the geodesic flow $\p$. Since hyperbolicity is  stable  under $C^1$ perturbations, it will imply that $\Omega_\psi$ is a hyperbolic set for $\psi$.\\
\indent In the definitions of these flows, they seem to be defined on different manifolds. We will start by giving a slightly different construction so that they live on the same manifold.\\
\indent Consider a path $\rho_u : \mathbb F_n \to \Diff(\Ss^1)$ for $u\in [0,1]$ defined as convex combinations of $\rho_0$ and $\rho_1$ (we chose free groups so that such a path can be easily defined). Recall the definition of $\Sigma_3$:
$$\Sigma_3= \{(x_-,x_0,x_+)\in (\Ss^1)^3 \vert x_-<x_0<x_+<x_-\}$$
\indent We can define an action of $\Gamma$ on $\Sigma_3\times \intff{0}{1}$ by:
$$ \gamma . (x_-,x_0,x_+,u) = (\rho_u(\gamma)(x_-),\rho_u(\gamma)(x_0),\rho_u(\gamma)(x_+),u)$$
\indent This action preserves the slices $\Sigma_3\times\{u\}$, which gives a map on the quotient $\pi : \Sigma_3 \times \intff{0}{1} / \Gamma  \to \intff{0}{1}$ which is a submersion. Each fiber $\pi^{-1}(\{u\})$ is diffeomorphic to the manifold $N_u$ on which the flow $\psi_u^t$ associated with the representation $\rho_u$  is defined.\\
\indent If $U\subset \Sigma_3$ is a relatively compact neighbourhood of $(L_{\rho_0(\Gamma)}\times \Ss^1\times L_{\rho_0(\Gamma)})\cap \Sigma_3$, then the restriction of $\pi$ to $U\times \intff{0}{1}$ is a proper submersion onto $\intff{0}{1}$, hence   a trivial fibration, i.e. there is a diffeomorphism $\Phi : U\times \intff{0}{1} / \Gamma \to N\times \intff{0}{1}$ such that projection on the second factor is equal to $\pi$. This shows that the flows $\psi_u$ (restricted to a  neighbourhood of the non wandering set) can be considered as flows on the manifold $N$, that vary continuously with $u$ in the $C^1$ topology. Therefore, if $\rho_1$ is sufficiently close to $\rho_0$, then $\Omega_{\psi_1}$ is a hyperbolic set for $\psi_1$.\\
\indent We will now use the notation $\psi$ for the flow associated to $\rho_1$, and $\alpha_1$ for the diagonal action of $\Gamma$ on $\Sigma_3$ (note that it is not exactly the same flow as defined in the proof of Theorem \ref{non_Fuchsian}, where we kept the action $\rho_0$ on the middle factor of $\Sigma_3$ so that the conjugacy with the geodesic flow would be differentiable along all the non wandering set).\\
\indent Given a volume $\omega_0$ on $N$, we set $\psi^{t*} \omega_0 = e^{-A(t,x)}\omega_0$. 
To find a volume $\omega_1=e^\sigma \omega_0$ that is invariant under $\psi$ at points of $\Omega_\psi$, we have to solve the equation $\sigma(\psi^t(x))-\sigma(x)=A(t,x)$ for all $x\in \Omega_\psi$. A necessary condition on the cocycle $A$ is that $A(T,x) = 0$ whenever $\psi^T(x)=x$. Livšic's Theorem states that this condition is sufficient.\\
\indent Let us show  that $A(T,x) = 0$ for periodic orbits $\psi^T(x)=x$. Since $A(T,x)=-\textrm{Log} \det(D\psi^T_x)$, we have to show that  the Jacobian  $\det(D\psi^T_x)$ is equal to $1$.\\
\indent To compute this Jacobian, we consider the lift $\tilde \psi^t$ to $\Sigma_3$, and $p:\Sigma_3\to \Sigma_3 /\Gamma$ the covering map. Since the flow $\tilde \psi^t$ is a reparametrization of the vector field $(0,1,0)$, it can be written:
$$\tilde \psi^t(x_-,x_0,x_+)=(x_-,f(t,x_-,x_0,x_+),x_+)$$
\indent If $\psi^T(x)=x$, then a lift $\tilde x=(x_-,x_0,x_+) \in p^{-1}(\{x\})$  satisfies $\tilde \psi^T(\tilde x)=\alpha_1(\gamma)( \tilde x)$ for some $\gamma \in \Gamma$. For all $y\in \Ss^1$ such that $(x_-,y,x_+)\in \Sigma_3$, we get $\tilde \psi^T(x_-,y,x_+)=(x_-,\rho_1(\gamma)(y),x_+)$, which shows that the matrix of $D\tilde \psi^T_{\tilde x}$ has the form:
\[ \left( \begin{array}{ccc}
1 & * & 0 \\
0 & \rho_1(\gamma)'(x_0) & 0\\
0 & * & 1 \end{array} \right)\]

\indent Consequently, its determinant is $\rho_1(\gamma)'(x_0)$. The derivative $D\psi^T_x$ is similar to $(D\alpha_1(\gamma)_{\tilde x})^{-1} D\tilde \psi^T_{\tilde x}$. The matrix of $D\alpha_1(\gamma)_{\tilde x}$ is the diagonal matrix:
\[ \left( \begin{array}{ccc}
\rho_1(\gamma)'(x_-) & 0 & 0 \\
0 & \rho_1(\gamma)'(x_0) & 0\\
0 & 0 & \rho_1(\gamma)'(x_+) \end{array} \right)\]

Since the action $\rho_1$ is spectrally Möbius-like and $x_-$ and $x_+$ are fixed points of $\rho_1(\gamma)$, we have $\rho_1(\gamma)'(x_-)\rho_1(\gamma)'(x_+)=1$, hence $\det(D\alpha_1(\gamma)_{\tilde x})=\rho_1(\gamma)'(x_0)$, and $\det(D\psi^T_x)   =1$.\\
\indent In order to apply Livšic's Theorem, one has to be precise on the exact setting, as well as on the required regularity. The first result, proved by Livšic in \cite{livsic}, concerns transitive Anosov flows, and deals with Hölder solutions. Smooth solutions for transitive Anosov flows are given  in \cite{LMM86}. Concerning compact topologically transitive hyperbolic sets, the existence of a Hölder-continuous  and even $C^1$ solutions can be found in \cite{KH} (Theorem 19.2.4 and Theorem 19.2.5). The main difficulty appears while studying  crossed derivatives for $C^2$ regularity. For smoothness outside of the Anosov setting (i.e. when the hyperbolic set is not the whole manifold), the only result concerns diffeomorphisms of surfaces in \cite{NT}. However, flows on three-manifolds are analogous to diffeomorphisms on surfaces.\\
\indent Lemma 3.3 of \cite{NT} states that there is a continuous solution $\sigma$ that is differentiable in restriction to stable and unstable leaves (\cite{NT} deals with diffeomorphisms of surfaces, but the same proof, up to replacing discrete sums by integrals, works for flows on three-manifolds). Going back to the cylinder $\mathcal C$, we get a function that is (uniformly) differentiable in restriction to leaves $\{x\} \times \Ss^1$ and $\Ss^1\times \{y\}$ for $x,y\in L_{\rho_1(\Gamma)}$. Theorem 1.5 of \cite{NT} implies that this solution is smooth on $\Ss^1\times L_{\rho_1(\Gamma)}\cup L_{\rho_1(\Gamma)}\times \Ss^1$ in the Whitney sense (i.e. that it is the restriction of a smooth function on $\mathcal C$).\\
\indent From there, Lemma \ref{vertical}, Proposition \ref{continuous} and Proposition \ref{smooth} show that $\rho_1$ is area-preserving.
\end{proof}

\subsection*{Acknowledgments} This work corresponds to Chapter 4 of my PhD thesis \cite{these}. I would like to thank my advisor Abdelghani Zeghib for his help throughout this work.

~\\
\footnotesize \textsc{UMPA, \'Ecole Normale Supérieure de Lyon, 46 allée d'Italie, 69364 Lyon Cedex 07,  France}\\
 \emph{E-mail address:}  \verb|daniel.monclair@ens-lyon.fr|

\end{document}